\documentclass[english]{amsart}
\usepackage{amsfonts,amssymb,amsmath,amsgen,amsthm}
\usepackage{hyperref,color}
\usepackage{pdfsync}
\usepackage{mathtools}

\makeatletter
\theoremstyle{plain}
\newtheorem{thm}{\protect\theoremname}
\theoremstyle{definition}
\newtheorem{defn}[thm]{\protect\definitionname}
\theoremstyle{remark}
\newtheorem{rem}[thm]{\protect\remarkname}
\theoremstyle{plain}
\newtheorem{lem}[thm]{\protect\lemmaname}
\theoremstyle{plain}
\newtheorem{prop}[thm]{\protect\propositionname}
\newtheorem{corollary}{Corollary}

\makeatother

\usepackage{babel}
\providecommand{\definitionname}{Definition}
\providecommand{\lemmaname}{Lemma}
\providecommand{\propositionname}{Proposition}
\providecommand{\remarkname}{Remark}
\providecommand{\theoremname}{Theorem}

\def\R{{\mathbf R}}
\def\N{{\mathbf N}}
\def\d{{\partial}}
\def\eps{\varepsilon}
\DeclareMathOperator{\RE}{Re}
\DeclareMathOperator{\IM}{Im}
\DeclareMathOperator{\diver}{div}

\numberwithin{equation}{section}

\date\today

\title[1D QHD]{Genuine Hydrodynamic Analysis to the 1-D QHD system: Existence, Dispersion and Stability}

\author{Paolo Antonelli}
\address{Gran Sasso Science Institute, viale Francesco Crispi, 7, 67100 L'Aquila}
\email{paolo.antonelli@gssi.it}

\author{Pierangelo Marcati}
\address{Gran Sasso Science Institute, viale Francesco Crispi, 7, 67100 L'Aquila}
\email{pierangelo.marcati@gssi.it}

\author{Hao Zheng}
\address{Gran Sasso Science Institute, viale Francesco Crispi, 7, 67100 L'Aquila}
\email{hao.zheng@gssi.it}


\begin{document}
\begin{abstract}
In this paper we consider a genuinely hydrodynamic approach for the one dimensional quantum hydrodynamics (QHD) system.  
In the recent years, the global existence of weak solutions with large data has been obtained in  \cite{AM1, AM2}, in several space dimensions, by using the connection between the hydrodynamic variables and the Schr\"odinger wave function. 
One of the main purposes of the present paper is to overturn this point of view, avoiding the need to postulate the a priori existence of a wave function that generates the hydrodynamic data.  In a first result, we are able to demonstrate the existence of finite energy weak solutions with large data.   Next we introduce a functional based on a new notion of generalized chemical potential, that allows us to prove the existence of more regular weak solutions under the assumption that the initial data have a finite generalized chemical potential and that the energy density is continuous on the boundaries of the vacuum set. 
These solutions enjoy an appropriate entropy inequality. We then obtain dispersive estimates for which the mass densities vanish and the speed formally tends to a rarefaction wave with a mechanism reminiscent of the Landau Damping.  We are also able to show, by using genuinely hydrodynamic arguments,  that for finite energy and finite mass weak solutions Morawetz-type estimates hold  and therefore by means of these estimates, in the case of solutions that have bounded generalized chemical potential and satisfies the entropy inequality, we get additional regularity.  
By virtue of these properties it is possible to obtain a result of strong stability in the energy norm, in the class of solutions with finite mass, finite energy and bounded generalized chemical potential.
Moreover we also analyze the possibility of extending the generalized chemical potential and a Radon measure that, in the atomic parts,  takes into account the potential differences on the phase boundaries.
\end{abstract}
\maketitle
\section{Introduction}\label{sect:intro}
This paper is concerned with the one dimensional quantum hydrodynamics (QHD) system
\begin{equation}\label{eq:QHD}
\left\{\begin{aligned}
&\d_t\rho+\d_xJ=0\\
&\d_tJ+\d_x\left(\frac{J^2}{\rho}\right)+\d_x p(\rho)=
\frac12\rho\d_x\left(\frac{\d^2_{x}\sqrt{\rho}}{\sqrt{\rho}}\right).
\end{aligned}\right.
\end{equation}

The unknowns $\rho$ and $J$ in \eqref{eq:QHD} represent the mass and momentum densities of the fluid, respectively.
$p(\rho)=\frac{\gamma-1}{\gamma}\rho^\gamma$ is the pressure term, which for convenience we assume to satisfy a $\gamma-$law with $1<\gamma<\infty$, although a more general pressure law can be taken into consideration. Contrarily to classical compressible fluids, here we do not need to assume convexity of the internal energy, see the Remark \ref{rem:gen_press} below. The term on the right hand side of the equation for the momentum density is a third order tensor  which represents a quantum effective potential and takes into account the quantum effects in the fluid. At the mathematical level, this term induces non negligible dispersive phenomena that change substantially the analysis with respect to classical newtonian fluids.
This system is used in physics to describe a compressible, inviscid fluid where quantum effects appear at a macroscopic scale and thus they need to be taken into account also in the hydrodynamical description. 
It is for instance the case when studying phenomena in superfluidity \cite{Khal}, Bose-Einstein condensation \cite{PS} or in the modeling of semiconductor devices at nanoscales \cite{Gar}.

In this paper we study the existence of weak solutions to \eqref{eq:QHD} in the energy space. Furthermore, we also show the stability of weak solutions in a regularity class determined by the introduction of a novel functional, see \eqref{eq:higher} below. In particular our stability result is given for arbitrarily large solutions, which are not required to be small perturbations of constant solutions. To our knowledge, this is the first result of this kind in such a class of solutions. In order to achieve the aforementioned result we will also present some a priori estimates which can also be analysed in different contexts, e.g. for general multidimensional Euler-Korteweg systems, see \eqref{eq:EK} below.

The energy associated to \eqref{eq:QHD}, formally conserved along the flow of solutions, reads
\begin{equation}\label{eq:en_QHD}
E(t)=\int_\R\frac12(\d_x\sqrt{\rho})^2+\frac{J^2}{2\rho}+f(\rho)\,dx,
\end{equation}
where the internal energy $f(\rho)$ is related to the pressure term by the relation 
$f(\rho)=\rho\int_0^\rho\frac{p(s)}{s^2}\,ds$.\\
The natural bounds given by \eqref{eq:en_QHD} imply that the only available control for the velocity field $v$ defined by $J=\rho v$ is in $L^2$ with respect to the measure $\rho\,dx$. In particular, there is no control of the velocity field in the vacuum region. For this reason, it turns out that, in order to deal with finite energy weak solutions to \eqref{eq:QHD}, it is more convenient to consider the unknowns $(\sqrt{\rho}, \Lambda)$, which define the hydrodynamic observables by $\rho=(\sqrt{\rho})^2$, $J=\sqrt{\rho}\Lambda$, see the Definitions \ref{def:FEWS} and \ref{def:FEWS_2} below for more details. More precisely, we will often refer to $(\sqrt{\rho}, \Lambda)$ as the hydrodynamic state variables of the system, while we will refer to the pair $(\rho, J)$ as the usual macroscopic hydrodynamic variables, which solve the system \eqref{eq:QHD}. \\
The system \eqref{eq:QHD} enjoys an interesting analogy with nonlinear Schr\"odinger equations, established through the Madelung transformations \cite{Mad}. Namely, given a wave function it is possible to associate to it a set of hydrodynamical variables. This analogy was rigorously set up and exploited in \cite{AM1, AM2} in order to prove the global existence of finite energy weak solutions to \eqref{eq:QHD} without any smallness or regularity assumptions. In particular, in \cite{AM1, AM2} the authors develop a polar factorization technique which overcome the difficulty of defining the velocity field in the vacuum region and allows to define the quantities $(\sqrt{\rho}, \Lambda)$ almost everywhere.\\
The first result in our paper shows that in the one dimensional case it is in fact possible to do also the opposite, namely given a set of finite energy hydrodynamical quantities $(\sqrt{\rho}, \Lambda)$ such that $\Lambda$ vanishes in the vacuum region, it is possible to define an associated wave function. Roughly speaking, if one thinks at the Madelung transform as a map from the space of (finite energy) wave functions to the space of (finite energy) hydrodynamic quantities, then we show that in 1D under some reasonable assumptions it is possible to invert this map. 
Moreover, we will show that it is possible to draw this analogy also in the framework of $H^2$ wave functions, by defining suitable hydrodynamic quantities, see Proposition \ref{prop:lift2} in Section \ref{sect:lift}.
Our arguments cannot be easily generalized to the higher dimensional case, as there one should need more structural assumptions on the hydrodynamical quantities, such as the generalized irrotationality condition (see Definition 1 in \cite{AM1} and subsequent Remark).
By inverting the Madelung trasform it is then possible to show that, given a set of finite energy initial data $(\sqrt{\rho_0}, \Lambda_0)$ such that $\Lambda_0$ vanishes on the vacuum region $\{\sqrt{\rho_0}=0\}$, there exists a global in time finite energy weak solution to the QHD system \eqref{eq:QHD} in one space dimension. 
In contrast with the results in \cite{AM1, AM2}, here we do not require the initial data to be Schr\"odinger-generated. As it will be clear in Section \ref{sect:lift}, in order to construct a weak solution we will anyway exploit the analogy with the Schr\"odinger dynamics, but no assumptions of this type need to be made on the initial data.\\
The second main problem addressed in this paper regards the stability of weak solutions. 
More precisely, we identify a class of weak solutions to \eqref{eq:QHD} enjoying a compactness property, namely for every sequence of weak solutions which satisfies suitable uniform bounds, it is possible to extract a subsequence converging to a weak solution. As previously said, to our knowledge this is the first stability result for system \eqref{eq:QHD} dealing with large, rough solutions.
If we restrict our analysis to Schr\"odinger-generated solutions, then by the polar factorization we could exploit a wide class of smoothing estimates available for nonlinear Schr\"odinger evolutions (e.g. Strichartz and Kato type estimates) in order to prove compactness of sequences of such solutions. On the other hand, under some regularity assumptions and by assuming the positivity of the mass density, it could be possible to apply relative entropy methods \cite{GT, GLT} or classical energy methods \cite{LM, JMR} to study the stability of solutions. However for arbitrary solutions to \eqref{eq:QHD} fewer estimates are available. In our paper we first provide a class of dispersive estimates yielding some information about the asymptotic behavior of solutions. As it will be also remarked later, such estimates hold true also for the multi dimensional case and even for the more general class of Euler-Korteweg systems.\\
Moreover, we introduce a novel functional, see \eqref{eq:higher} below that is (formally) uniformly bounded along the flow determined by \eqref{eq:QHD}. The set of states for which the functional is finite, will determine  the class of solutions for which we are able to prove our stability result.\\
To define the new functional, we start by considering the following chemical potential
\begin{equation}\label{eq:chem}
\mu=-\frac{1}{2}\frac{\partial_{x}^{2}\sqrt{\rho}}{\sqrt{\rho}}+\frac{1}{2}v^{2}+f'(\rho).
\end{equation} 
Formally it is possible to interpret $\mu$ as the first variation of the total energy functional with respect to the mass density \cite{GLT}, 
\[
\mu=\frac{\delta E}{\delta \rho}.
\]
As we will see later, the chemical potential $\mu$ cannot be used to carry out a satisfactory mathematical analysis in the framework of weak solutions to \eqref{eq:QHD_1d}. For this reason it will be more convenient to consider 
\begin{equation*}
\xi:= \rho\mu=-\frac14\d_{x}^2\rho+\frac12(\d_x\sqrt{\rho})^2+\frac12\Lambda^2+\rho f'(\rho).
\end{equation*}
The functional we are going to define below in fact gives a control on the quantity $\frac{\xi^2}{\rho}$, for this reason we are led to introduce a ``generalized chemical potential'' $\lambda\in L^2(\R)$, see also the Definition \ref{def:lambda} below, which formally equals $\lambda=\sqrt\rho\mu$ and is implicitly defined by
\begin{equation}\label{eq:lambda_intro}
\sqrt{\rho}\lambda=-\frac14\d_{x}^2\rho+e+p(\rho),
\end{equation}
where
\begin{equation}\label{eq:en_dens}
e=\frac12(\d_x\sqrt{\rho})^2+\frac12\Lambda^2+f(\rho)
\end{equation}
is the total energy density. A more rigorous definition of $\lambda$ will be given in Section \ref{sect:lambda}, see Proposition \ref{prop:lambda} therein.
The functional we consider in our study thus reads
\begin{equation}\label{eq:higher}
I(t)=\int\lambda^2+(\d_t\sqrt{\rho})^2\,dx.
\end{equation}
In this way we see that the finiteness of the functional $I(t)$  formally implies that the chemical potential is square integrable with respect to the quantum probability measure, namely $\mu\in L^2(\rho\,dx)$. However we stress again that the only object that is rigorously defined is $\lambda$, see Section \ref{sect:lambda} for more details.
\newline
Let us remark that for Schr\"odinger-generated hydrodynamical momenta, say $\rho=|\psi|^2$ and $J=\IM(\bar\psi\d_x\psi)$, the functional \eqref{eq:higher} actually equals
\begin{equation*}
I(t)=\int|\d_t\psi|^2\,dx.
\end{equation*}
Thus for NLS evolutions, we could think of the functional as providing a $H^2$ control for the wave function.

The quantities $\lambda$ and $\d_t\sqrt{\rho}$ formally allow to write the evolution equation for the energy density \eqref{eq:en_dens} in the case of smooth solutions, i.e.
\begin{equation}\label{eq:en_diff}
\d_te+\d_x(\Lambda\lambda-\d_t\sqrt{\rho}\d_x\sqrt{\rho})=0.
\end{equation}
In Section \ref{sect:apri} we will show that, by assuming that the following weak entropy inequality
\begin{equation}\label{eq:entr_ineq}
\d_te+\d_x(\Lambda\lambda-\d_t\sqrt{\rho}\d_x\sqrt{\rho})\le0
\end{equation}
holds in the distributional sense,
it is possible to infer a Morawetz type estimate which will yield an improved space-time control in $L^2$ for the energy density $e$, see  Proposition \ref{prop:unif0}.
Combining this bound together with the ones given by \eqref{eq:en_QHD} and \eqref{eq:higher} it is then possible to infer a compactness result for weak solutions to \eqref{eq:QHD}.

We can now present the main results of this paper. First of all we are going to provide a global existence result for finite energy initial data, which exploits the wave function lifting.
This allows us to remove the assumption on the initial data to be consistent with a wave function, which was needed in \cite{AM1, AM2}.

\begin{thm}[Global Existence of finite energy weak solutions]\label{thm:glob} 
Let us consider a pair of finite energy initial data $(\sqrt{\rho_0}, \Lambda_0)$,
\begin{equation}\label{eq:C1_intro}
\|\sqrt{\rho_0}\|_{H^1}+\|\Lambda_0\|_{L^2}\leq M_1
\end{equation}
and let us further assume that $\Lambda_0=0$ a.e. on the set $\{\rho_0=0\}$. Then there exists a global in time finite energy weak solution to the Cauchy problem \eqref{eq:QHD} which conserves the total energy for all times, $E(t)=E(0)$. In particular we have
\begin{equation}\label{eq:B1_1}
\|\sqrt{\rho}\|_{L^\infty(0, T;H^1(\R))}+\|\Lambda\|_{L^\infty(0, T;L^2(\R))}\leq M_1.
\end{equation}
\end{thm}
Furthermore, we are able to construct a family of more regular weak solutions. As we will see, the result we establish is not given in terms of the Sobolev regularity for $(\sqrt{\rho}, \Lambda)$ but rather in terms of the quantities involved in the analysis of the functional \eqref{eq:higher}.

Next Theorem deals with the global existence of weak solutions for initial data determined by GCP states (see the Definition \ref{def:GCP}) with continuous energy density.
\begin{thm}[Global existence of regular weak solutions]\label{thm:glob2}
Let us consider a pair of finite energy initial data $(\sqrt{\rho_0}, \Lambda_0)$ satisfying the same hypotheses of Theorem \ref{thm:glob} and let us further assume:
\begin{itemize}
\item $\d_x^2\rho_0\in L^1_{loc}(\R)$, $\d_x J_0\in L^1_{loc}(\R)$;
\item the initial energy density $e_0:=\frac12(\d_x\sqrt{\rho_0})^2+\frac12\Lambda_0^2+f(\rho_0)$ is continuous;
\item the initial data satisfy the bounds
\begin{equation}\label{eq:C2_intro}
\|(\frac{\Lambda_0^2}{\sqrt{\rho_0}}-\d_x^2\sqrt{\rho_0})\mathbf{1}_{\{\rho_0>0\}}\|_{L^2}
+\|\frac{\d_xJ_0}{\sqrt{\rho_0}}\mathbf{1}_{\{\rho_0>0\}}\|_{L^2}\leq M_2.
\end{equation}
\end{itemize}
Then the solution to \eqref{eq:QHD} constructed in Theorem \ref{thm:glob} further enjoys the following properties:
\begin{itemize}
\item there exists a measurable function $\lambda$ satisfying the identity \eqref{eq:lambda_intro} a.e. $(t, x)\in\R^2$, such that for any $0<T<\infty$,
\begin{equation}\label{eq:B2_1}
\|\d_t\sqrt{\rho}\|_{L^\infty(0, T;L^2(\R))}+\|\lambda\|_{L^\infty(0, T;L^2(\R))}\leq C(T, M_1, M_2);
\end{equation}
\item for any $0<T<\infty$ we have
\begin{equation}\label{eq:B2_c}
\|\rho\|_{L^\infty(0, T;H^2(\R))}+\|J\|_{L^\infty(0, T;H^1(\R))}+\|\sqrt{e}\|_{L^\infty(0, T;H^1(\R))}\leq C(T, M_1, M_2),
\end{equation}
where $e$ is the energy density defined in \eqref{eq:en_dens};
\item the solution satisfies the entropy identity \eqref{eq:en_diff} in the sense of distributions.
\end{itemize}
\end{thm}
	
\begin{rem}
The extra assumptions appearing in Theorem \ref{thm:glob2} allow us to exploit the wave function lifting for $H^2$ wave functions, see  Proposition \ref{prop:lift2}. As we will see later, those assumptions, and in particular \eqref{eq:C2_intro}, will be necessary to define the functional in 
\eqref{eq:higher}, which turns out to be uniformly bounded on compact time intervals for the solutions constructed in Theorem \ref{thm:glob2}. 
Conversely, it is also possible to show that the estimates given in \eqref{eq:B2_1} imply
\[
\|(\frac{\Lambda^2}{\sqrt{\rho}}-\d_x^2\sqrt{\rho})\mathbf{1}_{\{\rho>0\}}\|_{L^\infty_tL^2_x}
+\|\frac{\d_xJ}{\sqrt{\rho}}\mathbf{1}_{\{\rho>0\}}\|_{L^\infty_tL^2_x}\leq C(T,M_1,M_2).
\]
\end{rem}

The two Theorems above show the existence of global in time weak solutions to \eqref{eq:QHD}. However no uniqueness is provided here and in general such result does not hold true, see \cite{DFM} and \cite{MS}.
On the other hand, we are able to prove a stability result. 
More precisely we prove that the class of weak solutions satisfying the uniform bounds given by the energy and the functional in \eqref{eq:higher} enjoys a compactness property. In particular the weak solutions constructed in Theorem \ref{thm:glob2} belong to such class. Let us remark however that our result below holds for a general class of solutions, not necessarily those one constructed in Theorem \ref{thm:glob2}.
\begin{thm}[Stability]\label{thm:stab}
Let $0<T<\infty$ and let $\{(\rho_n, J_n)\}_{n\geq1}$ be a sequence of finite energy weak solutions to \eqref{eq:QHD} such that
\begin{itemize}
\item for any $n\in\N$ there exists $\lambda_n\in L^\infty(0, T;L^2(\R))$ satisfying the identity
\begin{equation}\label{eq:lambdan_intro}
\sqrt{\rho_n}\lambda_n=-\frac14\d_x^2\rho_x+e_n+p(\rho_n),
\end{equation}
where  the energy density $e_n$ is given by
\[
e_n=\frac12(\d_x\sqrt{\rho_n})^2+\frac12\Lambda_n^2+f(\rho_n);
\]
\item the following uniform bounds hold
\begin{equation}\label{eq:B1}
\|\sqrt{\rho_n}\|_{L^\infty(0, T;H^1(\R))}+\|\Lambda_n\|_{L^\infty(0, T;L^2(\R))}\leq M_1,
\end{equation}
\begin{equation}\label{eq:B2}
\|\d_t\sqrt{\rho_n}\|_{L^\infty(0, T;L^2(\R))}+\|\lambda_n\|_{L^\infty(0, T;L^2(\R))}\leq C(T, M_1, M_2),
\end{equation}
\end{itemize}
Moreover, let us also assume one  among the following conditions hold:
\begin{itemize}
\item[(1)] for almost every $t\in [0,T]$, $\rho_n(t,\cdot)>0$;
\item[(2)] for almost every $t\in [0,T]$, $e_n(t,\cdot)$ is continuous;
\item[(3)] $(\rho_n, J_n)$ satisfies the weak entropy inequality
\begin{equation*}
\d_t e_n+\d_x(\Lambda_n\lambda_n-\d_t\sqrt{\rho_n}\d_x\sqrt{\rho_n})\leq0,
\end{equation*} 
in the sense of distributions.
\end{itemize}
Then, up to subsequences, we have 
\begin{equation*}
\begin{aligned}
\sqrt{\rho_n}\to&\sqrt{\rho}&\textrm{in}\;L^2(0, T;H^1_{loc}(\R)),\\
\Lambda_n\to&\Lambda&\textrm{in}\;L^2(0, T;L^2_{loc}(\R)),
\end{aligned}
\end{equation*}
and $(\rho, J)=((\sqrt{\rho})^2, \sqrt{\rho}\Lambda)$ is a finite energy weak solution to \eqref{eq:QHD_1d}.
\end{thm}
The QHD system, together with similar evolutionary PDE models, is widely studied in the mathematical literature. Previous results can be found in \cite{JMR, LM, GM, JLM, HLM, HLMO}. System \eqref{eq:QHD} is also intimately related to the class of Euler-Korteweg fluids \cite{BG}, encoding capillary effects in the description; local and global analysis of small, regular perturbations of constant states are discussed in \cite{BGDD, AH}. The method of convex integration can be applied also to Euler-Korteweg system in order to show the existence of infinitely many solutions emanating from the same initial data \cite{DFM}, see also \cite{MS} where a non-uniqueness result for system \eqref{eq:QHD} is given by adopting a different strategy, related to the underlying wave function dynamics.
Recently also a class of viscous quantum fluid dynamical systems was considered. Such models can be derived from the Wigner-Fokker-Planck equation \cite{JLMG}, see also the review \cite{J}. The analysis of finite energy weak solutions for the quantum Navier-Stokes system was done in \cite{JqNS, AS, AS1}, see also \cite{LLX, LX} where similar arguments are used to study the compressible Navier-Stokes system with degenerate viscosity.

The contents of this paper is structured as follows: in Section \ref{sect:lift} we will prove the wave function lifting of a pair of hydrodynamical data $(\sqrt{\rho}, \Lambda)$ under some general assumptions. As a consequence, this will yield a global existence result. In Section \ref{sect:disp} we interpret the dispersive property and some a priori estimates on finite energy weak solutions to system \eqref{eq:QHD}.  In Section \ref{sect:apri} we provide some additional estimates for solutions with bounded generalized chemical potential and finally we show the stability of solutions in Section \ref{sect:comp}. We conclude our paper with Section \ref{sect:lambda}, where we provide a more detailed discussion on the function $\lambda$.

\section{Notations and Preliminaries}\label{sect:prel}

In this section we fix the notation that will be used through this paper.

We use the standard notation for Lebesgue and Sobolev norms

\[
||f||_{L_{x}^{p}}\coloneqq(\int_{\R}|f(x)|^{p}dx)^{\frac{1}{p}},
\]
\[
||f||_{W_{x}^{k,p}}\coloneqq\sum_{j=0}^k||\partial_{x}^{j}f||_{L_{x}^{p}},
\]
and let $H_{x}^{k}\coloneqq H^{k}(\R)$ denote the Sobolev space $=W^{k,2}(\R)$.
The mixed Lebesgue norm of functions $f:I\to L^{r}(\R)$ is
defined as 
\[
||f||_{L_{t}^{q}L_{x}^{r}}\coloneqq\left(\int_{I}||f(t)||_{L_{x}^{r}}^{q}dt\right)^{\frac{1}{q}}=\left(\int_{I}(\int_{\R}|f(x)|^{r}dx)^{\frac{q}{r}}dt\right)^{\frac{1}{q}},
\]
where $I\subset[0,\infty)$ is a time interval. Similarly the mixed
Sobolev norm $L_{t}^{q}W_{x}^{k,r}$ is defined. We use $C$ to denote the generic constant appearing in this paper, which may be written in the form $C(X)$ to indicate its dependence on the quantity $X$ and may change from line to line.

We recall here some basic properties of the following one-dimensional nonlinear Schr\"odinger equation
\begin{equation}\label{eq:NLS_1d}
\left\{\begin{aligned}
i\d_t\psi=&-\frac12\d_x^2\psi+|\psi|^{2(\gamma-1)}\psi\\
\psi(0)=&\psi_0,
\end{aligned}\right.
\end{equation}
with $\gamma\in(1, \infty)$.

The reader can find more details and the proofs of the statements of the following Theorem in the comprehensive monographs \cite{Caz, LP, Tao}.
\begin{thm}\label{thm:NLS}
Let $\psi_0\in H^1(\R)$ then there exists a unique global solution $\psi\in\mathcal C(\R; H^1(\R))$ to \eqref{eq:NLS_1d} such that the total mass and energy are conserved at all times, i.e.
\begin{equation}\label{eq:cons_NLS}
\begin{aligned}
M(t)=&\int_{\mathbb{R}}|\psi(t, x)|^2\,dx=M(0)\\
E(t)=&\int_{\mathbb{R}}\frac{1}{2}|\partial_{x}\psi|^{2}+\frac1\gamma|\psi|^{2\gamma}\,dx=E(0).
\end{aligned}\end{equation}
If moreover $\psi_0\in H^2(\R)$, then we also have $\psi\in\mathcal C(\R;H^2(\R))\cap\mathcal C^1(\R;L^2(\R))$ and for any $0<T<\infty$ 
\begin{equation}\label{eq:H2}
\|\psi\|_{L^\infty(0, T;H^2(\R))}+\|\d_t\psi\|_{L^\infty(0, T; L^2(\R))}\leq C(T, \|\psi_0\|_{H^2(\R)}).
\end{equation}
\end{thm}

Next we are going to recall the polar factorization technique developed in \cite{AM1, AM2}, see also the reviews \cite{AMDCDS, AM3}. This method allows to define the hydrodynamic quantities 
$(\sqrt{\rho}, \Lambda)$ and sets up a correspondence between the wave function dynamics and the hydrodynamical system.
The main advantage of this approach with respect to the usual WKB method for instance is that vacuum regions are allowed in the theory. Here we only give a brief introduction and state the results we will exploit later, for a more detailed presentation and for a proof of the statements in Lemma \ref{lemma:polar} below we address to Section 3 in \cite{AHMZ}.

Given any function $\psi\in H^1(\Omega)$, where $\Omega\subset\R$ is an open set, we can define the set of polar factors as
\begin{equation}\label{eq:set_pol}
P(\psi)\coloneqq\left\{ \phi\in L^\infty(\Omega);\ \|\phi\|_{L^{\infty}}\leq1,\ \sqrt{\rho}\phi=\psi\ a.e.\right\} ,
\end{equation}
where $\sqrt{\rho}:=|\psi|$. 
In general this set can contain more than one element, due to the possible appearance of vacuum regions. Nevertheless, as it will be shown in the next Lemma, the hydrodynamical quantities $(\sqrt{\rho}, \Lambda)$ are well defined and they furthermore enjoy suitable stability properties.

\begin{lem}[Polar factorization \cite{AM1, AM2}]\label{lemma:polar}
Let $\psi\in H^{1}(\Omega)$ and $\sqrt{\rho}:=|\psi|$, and let $\phi\in P(\psi)$. Then we have $\d_x\sqrt{\rho}=\RE(\bar\phi\d_x\psi)\in L^2(\Omega)$ and if we set 
$\Lambda:= Im(\overline{\phi}\partial_{x}\psi)$, we have 
\begin{equation}\label{eq:quad}
|\partial_{x}\psi|^{2}=(\partial_{x}\sqrt{\rho})^{2}+\Lambda^{2},\quad\textrm{a.e.}\;x\in\Omega.
\end{equation}
Moreover if $\{\psi_n\}\subset H^1(\Omega)$ satisfies $\|\psi_n-\psi\|_{H^1(\Omega)}\to0$ as $n\to\infty$, then we have
\begin{equation*}
\d_x\sqrt{\rho_n}\to\d_x\sqrt{\rho},\quad\Lambda_n\to\Lambda,\quad{\rm in}\;L^2,
\end{equation*}
with $\sqrt{\rho_n}:=|\psi_n|, \Lambda_n:= \IM(\bar\phi_n\d_x\psi_n)$, $\phi_n$ being a unitary factor for $\psi_n$.
\end{lem}

We stress here the fact that $\Lambda=\IM(\bar\phi\d_x\psi)$ is well-defined even if $\phi$ is not uniquely determined, this is a consequence of Theorem 6.19 in \cite{LL}, which we will state below as it will be thoroughly used in our exposition.
\begin{lem}\label{lemma:LL}
Let $g:\Omega\to\mathbb{R}$ be in $H^{1}(\Omega)$, and 
\[
B=g^{-1}(\{0\})=\left\{ x\in\Omega\ :\ g(x)=0\right\} .
\]

Then $\nabla g(x)=0$ for almost every $x\in B$.
\end{lem}
This is in fact true in any space dimensions, but in the 1-D case it trivializes. Since we consider finite energy
solutions in one-dimensional space, Sobolev embedding implies that $\psi(t,x)$
is continuous, and $f^{-1}(\{0\})=\Omega_{0}$ is a closed set in $\mathbb{R}$,
which means it is a countable disjoint union of closed
intervals and points. It is straightforward to see that $\{x\in\Omega_{0}\ :\ \partial_{x}\psi(x)\ne0\}\subset\partial\Omega_{0}$,
and in the 1-dimensional case $\partial\Omega_0$ is countable, so the
conclusion of Lemma \ref{lemma:LL} is naturally true.

By combining the well-posedness result for the NLS equation \eqref{eq:NLS_1d} stated in Theorem \ref{thm:NLS} and the polar factorization method recalled in Lemma \ref{lemma:polar} 
we can prove an existence result for finite energy weak solutions to \eqref{eq:QHD}, see Proposition \ref{prop:old_1d} below. This result is the one dimensional analogue of some of the results proved in \cite{AM1, AM2} for the two and three dimensional cases, see also the review \cite{AHMZ}.

Before stating the results let us first recall a useful identity regarding the nonlinear dispersive term on the right hand side of the equation for the momentum density in \eqref{eq:QHD}, namely it can also be written as
\begin{equation}\label{eq:bohm}
\frac12\rho\d_x\left(\frac{\d_{x}^2\sqrt{\rho}}{\sqrt{\rho}}\right)=\frac14\d_{x}^3\rho-\d_x[(\d_x\sqrt{\rho})^2].
\end{equation}
By using identity \eqref{eq:bohm}, system \eqref{eq:QHD} then becomes
\begin{equation}\label{eq:QHD_1d}
\left\{\begin{aligned}
&\d_t\rho+\d_xJ=0\\
&\d_tJ+\d_x[\Lambda^2+p(\rho)+(\d_x\sqrt{\rho})^2]=\frac14\d_{x}^3\rho.
\end{aligned}\right.
\end{equation}
We can now give the definition of finite energy weak solutions.

\begin{defn}[Weak solutions]\label{def:FEWS}
Let $\rho_0, J_0\in L^1_{loc}(\R)$, we say the pair $(\rho, J)$ is a weak solution to the Cauchy problem for \eqref{eq:QHD} with initial data
$
\rho(0)=\rho_0$, $J(0)=J_0,
$
in the space-time slab $[0, T)\times\R$ if there exist two locally integrable functions 
$
\sqrt{\rho}\in L^2_{loc}(0, T;H^1_{loc}(\R))$, $\Lambda\in L^2_{loc}(0, T;L^2_{loc}(\R))
$
such that
\begin{itemize}
\item[(i)] $\rho:=(\sqrt{\rho})^2, J:=\sqrt{\rho}\Lambda$;
\item[(ii)] for any $\;\eta\in\mathcal C^\infty_0([0, T)\times\R)$,
\begin{equation}\label{eq:QHD_cty}
\int_0^T\int_{\R}\rho\d_t\eta+J\d_x\eta\,dxdt+\int_{\R}\rho_0(x)\eta(0, x)\,dx=0;
\end{equation}
\item[(iii)] for any $\;\zeta\in\mathcal C^\infty_0([0, T)\times\R)$,
\begin{equation}\label{eq:QHD_mom}
\begin{aligned}
\int_0^T\int_{\R}&J\d_t\zeta+(\Lambda^2+p(\rho)+(\d_x\sqrt{\rho})^2)\d_x\zeta
-\frac{1}{4}\rho\d_x^3\zeta\,dxdt\\
&+\int_{\R}J_0(x)\zeta(0, x)\,dx=0.
\end{aligned}
\end{equation}
\end{itemize}
We say $(\rho, J)$ is a global in time weak solution if we can take $T=\infty$ in the above definition.
\end{defn}

\begin{defn}[Finite mass and finite energy weak solutions]\label{def:FEWS_2}
Let $(\rho,J)$ be a weak solution to the QHD system \eqref{eq:QHD} as in the Definition \ref{def:FEWS}. Let $M(t)$ be the total mass of $(\rho,J)$,
\begin{equation*}
M(t)=\int\rho(t, x)\,dx,
\end{equation*}
and $E(t)$ be the total energy denoted by
\begin{equation*}
E(t)=\int_{\R}\frac12|\Lambda|^2+\frac12|\d_x\sqrt{\rho}|^2+f(\rho)\,dx.
\end{equation*}
We say $(\rho,J)$ is a finite mass weak solution, if for almost every $t\in[0, T)$ we have $M(t)\leq M(0)$. Analogously, $(\rho,J)$ is called a finite energy weak solution if for almost every $t\in[0, T)$, we have $E(t)\leq E(0)$.
\end{defn}

\begin{rem}
Let $(\rho, J)$ be a weak solution with finite initial mass, i.e. $M(0)<\infty$, then it is straightforward to see that the total mass $M(t)$ is conserved for almost every time, see for example Theorem 1.3.4 in \cite{Daf}. A similar argument can also be applied to the total momentum
\begin{equation*}
P(t)=\int J(t, x)\,dx,
\end{equation*}
which is conserved for weak solutions with finite mass and finite energy.
\end{rem}
In our paper we will always consider weak solutions with finite mass and energy. 
Moreover, in order to simplify the exposition, we will always write finite energy weak solutions to mean weak solutions with both finite energy and mass.

For initial data generated by wave functions, the existence of finite energy weak solutions was proved in \cite{AM1, AM2}. Here we state the analogous result for the  one dimensional QHD system, which is completely similar to its higher dimensional counterpart. For this reason we address the interested reader to \cite{AM1, AM2}, see also Theorem 4.2. in \cite{AHMZ}.
\begin{prop}\label{prop:old_1d}
Let $(\rho_0, J_0)$ be such that $\rho_0=|\psi_0|^2$ and $J_0=\IM(\bar\psi_0\d_x\psi_0)$, for some $\psi_0\in H^1(\R)$. Then there exists a global in time finite energy weak solution to \eqref{eq:QHD_1d} such that $\sqrt{\rho}\in L^\infty(\R;H^1(\R))$, $\Lambda\in L^\infty(\R;L^2(\R))$ and the total energy is conserved for all times.
\end{prop}

We conclude this section by giving a preliminary discussion about the higher order functional introduced in \eqref{eq:higher} and a definition for the quantity $\lambda$ involved, formally related to the chemical potential $\mu$ by $\lambda=\sqrt{\rho}\mu$. A more detailed discussion about $\lambda$ can be found in Section \ref{sect:lambda}.

As mentioned in the Introduction, to obtain the stability properties of weak solutions, we will exploit the a priori bounds inferred from the functional 
\[
I(t)=\int_\R \lambda^2+(\d_t\sqrt\rho)^2dx,
\]
where $\lambda\in L^\infty(0,T;L^2(\R))$ is implicitly given by
\begin{equation}\label{eq:xi_def}
\xi=\sqrt\rho\lambda=-\frac14\d_x^2\rho+\frac12(\d_x\sqrt\rho)^2+\frac12\Lambda^2+f'(\rho)\rho,
\end{equation}
where $f(\rho)=\rho^\gamma/\gamma$ is the internal energy, see \eqref{eq:en_QHD}.
The functional $I(t)$ has a direct interpretation for Schr\"odinger-generated solutions. Indeed let us assume for the moment that we have $\rho=|\psi|^2$ and $J=\IM(\bar\psi\d_x\psi)$, for some solutions $\psi$ to \eqref{eq:NLS_1d}. Then we can see that the functional \eqref{eq:higher} actually equals
\begin{equation*}
I(t)=\int_\R|\d_t\psi|^2\,dx.
\end{equation*}
Indeed by using the polar factorization we have 
\[
\d_t\sqrt\rho=\RE(\bar\phi\d_t\psi),\quad \lambda=-\IM(\bar\phi\d_t\psi),
\]
where $\phi\in P(\psi)$ is a polar factor. On the other hand, by invoking Lemma \ref{lemma:LL}, we see that for Schr\"odinger-generated functions, $\lambda$ vanishes almost everywhere on the vacuum region $\{\rho=0\}$. 
Let us remark that in this case identity \eqref{eq:xi_def} is recovered from $\xi=\sqrt{\rho}\lambda=-\IM(\bar\psi\d_t\psi)$ and by  exploiting the fact that $\psi$ is a solution to \eqref{eq:NLS_1d}.
In Section \ref{sect:lambda} we provide a more detailed characterization for $\lambda$, in particular it will be given independently of the dynamics. For this reason, given an arbitrary hydrodynamic state $(\sqrt{\rho}, \Lambda)$, the properties of the generalized chemical potential $\lambda$ will be characterized in terms of $(\sqrt{\rho}, \Lambda)$ and their space derivatives. In what follows we provide a rigorous definition for $\lambda$, holding for arbitrary hydrodynamic states, not necessarily Schr\"odinger-generated.
\begin{defn}\label{def:lambda}
Let $(\sqrt{\rho}, \Lambda)\in H^1(\R)\times L^2(\R)$, such that $\Lambda=0$ a.e. on $\{\rho=0\}$ and $\d_x^2\rho\in L^1_{loc}(\R)$.
Then we define the generalized chemical potential $\lambda$ to be the measurable function given by
\begin{equation}\label{eq:def_lambda}
\lambda=\left\{\begin{array}{cc}
-\frac12\d_{x}^2\sqrt{\rho}+\frac12\frac{\Lambda^2}{\sqrt{\rho}}+f'(\rho)\sqrt{\rho}&\textrm{in }\;\{\rho>0\}\\
0&\textrm{elsewhere}
\end{array}\right.
\end{equation}
\end{defn}
In what follows we provide the definition of bounded generalized chemical potential hydrodynamic states, namely those states where $\lambda$ can be rigorously characterized, see Section \ref{sect:lambda}.

\begin{defn}[GCP states]\label{def:GCP}
We say that the pair $(\sqrt{\rho}, \Lambda)$ is a state with bounded generalized chemical potential (GCP state, in short) if:
\begin{itemize}
\item $\sqrt{\rho}\in H^1(\R)$ and $\Lambda\in L^2(\R)$;
\item $\Lambda=0$ a.e. on $\{\rho=0\}$;
\item $\d_x^2\rho, \d_xJ\in L^1_{loc}(\R)$, where $(\rho,J)=((\sqrt\rho)^2,\sqrt\rho\Lambda)$;
\item the following bounds are satisfied
\begin{equation*}
\|\mathbf{1}_{\{\rho>0\}}\d_x J/\sqrt{\rho}\|_{L^2(\R)}+\|\lambda\|_{L^2(\R)}\leq C,
\end{equation*}
where $\lambda$ is defined as in \eqref{eq:def_lambda}.
\end{itemize}
\end{defn}
The notion of GCP states naturally yields a regularity class for weak solutions. As it will be clear, this is indeed the class of solutions for which Theorem \ref{thm:stab} applies.
Let us remark that the stability property shown in Theorem \ref{thm:stab} is independent of the rigorous characterization of $\lambda$ presented in Section \ref{sect:lambda}. Indeed for our stability result we will only need that, associated to $(\sqrt{\rho}, \Lambda)$, there exists a $\lambda\in L^\infty_tL^2_x$, uniformly bounded and satisfying identity \eqref{eq:lambdan_intro}.

\begin{defn}[GCP solutions]\label{def:cptsln}
Let $(\rho, J)$ be a finite energy weak solution to \eqref{eq:QHD_1d} on 
$[0,T]\times\R$. We say that $(\rho, J)$ is a GCP solution for the system \eqref{eq:QHD_1d} if and only if
\begin{itemize}
\item there exists $\lambda\in L^\infty(0, T;L^2(\R))$, such that
\begin{equation*}
\sqrt{\rho}\lambda=-\frac14\d_{x}^2\rho+e+p(\rho),
\end{equation*}
where the energy density $e$ is defined by
\begin{equation*}
e=\frac12(\d_x\sqrt{\rho})^2+\frac12\Lambda^2+f(\rho);
\end{equation*}
\item for some $0<M_1,M_2<\infty$, the following bounds are satisfied
\begin{equation*}
\begin{aligned}
\|\sqrt{\rho}\|_{L^\infty(0, T; H^1(\R))}+\|\Lambda\|_{L^\infty(0, T; L^2(\R))}\leq& M_1,\\
\|\d_t\sqrt{\rho}\|_{L^\infty(0, T; L^2(\R))}+\|\lambda\|_{L^\infty(0, T; L^2(\R))}\leq& M_2.
\end{aligned}
\end{equation*}
\end{itemize}
\end{defn}
\noindent
Moreover, in our analysis it will be also important to exploit the weak entropy inequality \eqref{eq:entr_ineq}.

\begin{defn}\label{def:entrsln}
Let $(\rho,J)$ be a weak solution to \eqref{eq:QHD_1d} on $[0, T)\times\R$. We say that $(\rho, J)$ satisfies the weak entropy inequality if 
\begin{itemize}
\item $e\in L^1_{loc}([0, T)\times\R)$;
\item $\lambda, \d_t\sqrt{\rho}\in L^2_{loc}([0, T)\times\R)$;
\item the following inequality
\begin{equation*}
\d_te+\d_x(\Lambda\lambda-\d_t\sqrt{\rho}\d_x\sqrt{\rho})\leq 0,
\end{equation*}
is satisfied in the sense of distributions on $[0, T)\times\R$.
\end{itemize}
\end{defn}

\section{Wave Function Lifting and Global Existence}\label{sect:lift}
In the previous section, see Lemma \ref{lemma:polar}, we saw that given a wave function $\psi\in H^1$, by means of the polar factorization it is possible to determine suitable hydrodynamical states $(\sqrt{\rho}, \Lambda)$. 
In this section we are going to prove the converse, namely given $(\sqrt{\rho}, \Lambda)$ with finite energy such that $\Lambda=0$ a.e. in the vacuum region $\{\rho=0\}$, it is possible to define an associated wave function.
The assumption that $\Lambda$ vanishes on the vacuum is quite reasonable in view of the polar factorization and of Lemma \ref{lemma:LL}. Indeed for $\psi\in H^1$, we have $\d_x\psi=0$ a.e. on $\{\rho=0\}$ and consequently the quantity $\Lambda\in L^2$ constructed in Lemma \ref{lemma:polar} satisfies $\Lambda=0$ a.e. in the vacuum region.

We will then exploit the results below in order to prove Theorem \ref{thm:glob} on the existence of global solutions to \eqref{eq:QHD_1d}. 
\begin{defn}\label{def:ass}
Let $\Omega\subset\R$ be an arbitrary open set. Given the hydrodynamic state $(\sqrt\rho,\Lambda)\in H^1(\Omega)\times L^2(\Omega)$ we say that the wave function $\psi\in H^1(\Omega)$ is associated to $(\sqrt{\rho}, \Lambda)$ if
\begin{equation}\label{eq:polar}
\sqrt\rho=|\psi|\quad \textrm{and} \quad \Lambda=\IM(\bar\phi\d_x\psi),
\end{equation}
where $\phi\in P(\psi)$ is a polar factor of $\psi$. 
\end{defn}
The following Proposition gives sufficient and necessary conditions for a pair of hydrodynamic data $(\sqrt\rho,\Lambda)$ to have an associated wave function $\psi\in H^1(\R)$.

\begin{prop}\label{prop:lift1}
Let $(\sqrt\rho,\Lambda)$ be a pair of hydrodynamic data. There exists an associated wave function $\psi\in H^1(\R)$ in the sense of the Definition \ref{def:ass}, if and only if $(\sqrt\rho,\Lambda)$ satisfies 
\begin{itemize}
\item[(1)] there exists a constant $0<M_1<\infty$ such that $\|\sqrt{\rho}\|_{H^1(\R)}+\|\Lambda\|_{L^2(\R)}\leq M_1$;
\item[(2)] $\Lambda=0$ a.e. $x$ in $\{\rho=0\}$.
\end{itemize}
Furthermore we have 
\[
\d_x\psi=(\d_x\sqrt\rho+i\Lambda)\phi,\quad \|\psi\|_{H^1(\R)}\leq C(M_1),
\]
where $\phi\in P(\psi)$ is a polar factor of $\psi$.
\end{prop}
\begin{rem}
Conditions (1) and (2) in the statement of the Proposition above give necessary and sufficient conditions for the existence of an $H^1$ function $\psi$ associated to $(\sqrt{\rho}, \Lambda)$. Note however that such a wave function is not necessarily unique; this will be further discussed later in this section.
\end{rem}

\begin{proof}
Let $\psi\in H^1(\R)$ be associated to $(\sqrt{\rho}, \Lambda)$, then by means of the polar factorization in Lemma \ref{lemma:polar} we have that both conditions (1) and (2) are fulfilled. 

To prove the converse statement, we consider a sequence $\{\delta_n\}$ of Schwartz functions such that $\delta_n(x)>0$ for all $x\in\R$ and $\delta_n\to0$ as $n\to\infty$. For instance we may consider $\delta_n(x)=\frac1ne^{-|x|^2/2}$. We define the following approximating hydrodynamical quantities
\begin{equation*}
\sqrt{\rho_n}=\sqrt{\rho}+\delta_n, \quad\Lambda_n=\frac{J}{\sqrt{\rho_n}}=\frac{\sqrt{\rho}}{\sqrt{\rho_n}}\Lambda,
\end{equation*}
then we can check that also the sequence $\{(\sqrt{\rho_n}, \Lambda_n)\}$ satisfies a uniform bound as in condition (1). Indeed,
\begin{equation*}
\|\sqrt{\rho_n}\|_{H^1(\R)}\leq\|\sqrt{\rho}\|_{H^1(\R)}+\|\delta_n\|_{H^1(\R)}\leq C(M_1)
\end{equation*}
and since $\sqrt{\rho_n}(x)\geq\sqrt{\rho}(x)$ for every $x\in\R$, we have $|\Lambda_n(x)|\leq|\Lambda(x)|$ a.e. and consequently 
$\|\Lambda_n\|_{L^2(\R)}\leq\|\Lambda\|_{L^2(\R)}$. By construction it is straightforward that $\sqrt{\rho_n}\to\sqrt{\rho}$ in $H^1(\R)$. Moreover, by assumption we have $\Lambda(x)=0$ a.e. $x$ in the set $\{\rho=0\}$ and by definition of the approximants the same holds also for $\Lambda_n$. Hence $\sqrt{\rho_n}(x)\to\sqrt{\rho}(x)$ and $\Lambda_n(x)\to\Lambda(x)$ a.e. in $\R$.
Then the dominated convergence theorem yields $\Lambda_n\to\Lambda$ in $L^2(\R)$.
Furthermore, since $\sqrt{\rho_n}(x)>0$ it is possible to define the approximating velocity field
\begin{equation*}
v_n=\frac{\Lambda_n}{\sqrt{\rho_n}}.
\end{equation*}
Notice that, since by definition $\sqrt{\rho_n}$ is uniformly bounded away from zero on compact intervals, we have $v_n\in L^1_{loc}(\R)$, hence it makes sense to define the phase
\begin{equation*}
S_n(x)=\int_0^xv_n(x)\,dx
\end{equation*}
and consequently the wave function 
\begin{equation*}
\psi_n(x)=\sqrt{\rho_n}(x)e^{iS_n(x)}.
\end{equation*}
Let us remark that $\psi_n$ is uniquely defined, up to a constant phase shift. We now show that the sequence $\psi_n$ has a limit $\psi\in H^1(\R)$ satisfying $|\psi|=\sqrt{\rho}$,
$\IM(\bar\phi\d_x\psi)=\Lambda$, where $\phi\in P(\psi)$. Since
$\d_x\psi_n=e^{iS_n}(\d_x\sqrt{\rho_n}+i\Lambda_n)$, we have
\begin{equation*}
\|\psi_n\|_{H^1}^2=\|\d_x\sqrt{\rho_n}\|_{H^1}^2+\|\Lambda_n\|_{L^2}^2\leq C(M_1),
\end{equation*}
thus, up to subsequences, $\psi_n\rightharpoonup\psi$ in $H^1$. On the other hand, we also have $\sqrt{\rho_n}\to\sqrt{\rho}$ in $H^1$, $\Lambda_n\to\Lambda$ in $L^2$ and moreover $e^{iS_n}\rightharpoonup\phi$ weakly* in $L^\infty$, for some $\phi\in L^\infty$. It is straightforward to check that $\phi\in P(\psi)$, indeed we have $\psi_n=\sqrt{\rho_n}e^{iS_n}\rightharpoonup\sqrt{\rho}\phi=\psi$ in $L^2(\R)$. Furthermore
\begin{equation*}
\d_x\psi_n=e^{iS_n}\left(\d_x\sqrt{\rho_n}+i\Lambda_n\right)\rightharpoonup\phi\left(\d_x\sqrt{\rho}+i\Lambda\right), \quad\textrm{in}\;L^2(\R)
\end{equation*}
so that $\d_x\psi=\phi\left(\d_x\sqrt{\rho}+i\Lambda\right)$ and hence $(\sqrt{\rho}, \Lambda)$ are the hydrodynamical quantities associated to $\psi$ in the sense of the Definition \ref{def:ass}.
\end{proof}

We should emphasize that the wave function lifting given in Proposition \ref{prop:lift1} can not enjoy suitable stability properties in $H^1(\R)$, as opposed to the polar factorization given in Lemma \ref{lemma:polar}. Actually the wave function associated to a pair of finite energy hydrodynamic data can not be uniquely determined because of the arbitrary phase shifts allowed on each connected component of $\{\rho>0\}$, see the Remark \ref{rmk:non_uniq} below. In what follows we discuss some stability/instability properties of the wave function lifting. First of all, we show that on any connected component of the set $\{\rho>0\}$ the only source of non-uniqueness for the wave function $\psi$ associated to $(\sqrt{\rho}, \Lambda)$ is given by phase shifts.

\begin{lem}\label{lemma:rot}
Let $I\subset\R$ be an arbitrary interval and $\psi_1,\ \psi_2\in H^1(I)$ be two wave functions associated to the same hydrodynamic data $(\sqrt{\rho},\Lambda)\in H^1(I)\times L^2(I)$. Furthermore we assume that $\rho>0$ on the interval $I$. 

Then there exists a constant phase shift $\theta\in [0,2\pi)$, such that
\begin{equation}\label{eq:lemma16}
\psi_2=e^{i\theta}\,\psi_1.
\end{equation}
\end{lem}

\begin{proof}
By our assumption and the polar factorization in Lemma \ref{lemma:polar}, we have 
\[
\psi_j=\sqrt\rho\phi_j\quad and\quad \d_x\psi_j=(\d_x\sqrt{\rho}+i\Lambda)\phi_j,\quad j=1,2,
\]
where $\phi_j\in P(\psi_j)$, $j=1,2$.
Since $\rho>0$ on the interval $I$, the polar factor $\phi_j$ is uniquely defined by
\[
\phi_j=\frac{\psi_j}{\sqrt\rho},
\]
then we can define function $f=\psi_1/\psi_2\in H^1(I)$. A direct computation shows that
\begin{align*}
\d_x f=\frac{1}{\psi_2^2}\left[\sqrt\rho\phi_2(\d_x\sqrt{\rho}+i\Lambda)\phi_1-\sqrt\rho\phi_1(\d_x\sqrt{\rho}+i\Lambda)\phi_2\right]=0,\quad a.e.\ \textrm{in}\; I.
\end{align*}
Since the interval $I$ is connected, the function $f=\psi_1/\psi_2$ is a constant $C$ on $I$,  which satisfies
\[
|C|=\frac{|\psi_1|}{|\psi_2|}=1,
\]
therefore formula \eqref{eq:lemma16} holds true.
\end{proof}

\begin{rem}\label{rmk:psiconv}
We now show that the stability in the hydrodynamic variables in general does not imply $H^1$ stability of the associated wave function. For simplicity we consider the problem on the interval $[-1,1]$, and let us consider the hydrodynamic quantities $\sqrt{\rho_0}=|x|$, $\Lambda_0=0$ and the wave function
\[
\psi_0(x)=e^{iH(x)\frac{\pi}{2}}|x|,
\]
where $H(x)$ is the Heaviside function.
It is straightforward to show that $\psi_0\in H^1([-1,1])$ is a wave function associated to $(\sqrt{\rho_0},\Lambda_0)=(|x|, 0)$. On the other hand, we consider the approximating hydrodynamic quantities $\sqrt{\rho_n}=(|x|+\frac1n)$ and $\Lambda_n=0$, then by Lemma \ref{lemma:rot} the only associated wave functions are given by 
\[
\psi_n=e^{i\theta_n}(|x|+1/n),
\]
where $\theta_n\in [0,2\pi)$. 
Now if $\{\tilde\psi_n\}$ is a Cauchy sequence in $H^1$, it follows that $\{e^{i\theta_n}\}$ is a Cauchy sequence in the unit circle $\mathcal{S}^1$. Therefore $\tilde\psi_n$ converges in $H^1$ to $\tilde\psi_0=e^{i\theta}|x|$ with $\theta\in[0,2\pi)$, which is also a wave function associated to $(\sqrt{\rho_0},\Lambda_0)$. However $\tilde\psi_0\ne\psi_0$ for any $\theta\in[0,2\pi)$. 
\end{rem}

It is straightforward to notice that the source of instability given by the previous example stems from the combination of two facts. The former one is the indetermination of the lifted wave function up to phase shift on connected components of $\{\rho>0\}$, whereas the latter is the change of the number of connect components of the non-vacuum regions in the limit.
Let us remark that a similar kind of instability has also been used in \cite{MS} to show non-uniqueness of solutions for the system \eqref{eq:QHD}.
On the other hand in the following Proposition \ref{prop:psiconv} we can show that there is at least one wave function associated to the limiting hydrodynamic data, which can be attained as a strong limit in $H^1(\R)$ of a subsequence of associated wave functions. Furthermore in Proposition \ref{prop:lift2} we will prove the wave function lifting at $H^2$ level. More precisely for hydrodynamic data satisfying suitable regularity assumption, we can specify a wave function in $H^2(\R)$ associated to it, which has well-designed phase shifts on all components of $\{\rho>0\}$, in such a way that $\d_x\psi$ has no phase jump at vacuum boundaries.

Before stating Propositions \ref{prop:psiconv} and Proposition \ref{prop:lift2}, here we collect some elementary properties of Sobolev functions that will be used later. Let $g$ be a function in $H^1(a,b)$ and $H^1(b,c)$, then $g\in H^1(a,c)$ if and only if $g$ is continuous at point $b$, and in this case we have
\[
\|\d_x g\|_{L^2(a,c)}^2=\|\d_x g\|_{L^2(a,b)}^2+\|\d_x g\|_{L^2(b,c)}^2.
\]
Then let us consider $g\in H^1(\Omega)$, where $\Omega\subset\R$ is an open set. The continuity of $g$ allows us to decompose the set $\{x;g(x)\ne0\}$ into disjoint open intervals, i.e. 
\begin{equation}\label{eq:poscomp}
\{x;g\ne0\}=\cup_{j}(a_j,b_j), \quad g(a_j)=g(b_j)=0.
\end{equation}
The next Lemma shows that we can introduce arbitrary constant phase shift on each component $(a_j,b_j)$, without breaking the $H^1$ regularity of $g$.

\begin{lem}\label{lemma:H1ext}
Let $g\in H^1(\Omega)$ and let $\Theta\in L^\infty(\Omega)$ be a piecewise constant phase shift given by the formula
\begin{equation}\label{eq:theta}
\Theta=\exp\left(i\sum_j\theta_j\mathbf{1}_{(a_j,b_j)}\right),\quad \theta_j\in[0,2\pi),
\end{equation}
where $(a_j,b_j)$'s are the components of $\{g\ne 0\}$ as in \eqref{eq:poscomp}. 
Then we have $\Theta g\in H^1(\Omega)$, and 
\begin{equation}\label{eq:dtheta}
\d_x (\Theta g)=\Theta \d_x g.
\end{equation} 
\end{lem}

\begin{proof}
The proof of the Lemma follows a standard argument of weak derivative. We take $\eta\in C_c^\infty(\Omega)$ to be a test function and consider the weak derivative $\d_x(\Theta g)$:
\begin{align*}
\int_\Omega\eta\d_x (\Theta g)\,dx=&-\int_\Omega \Theta g\d_x\eta\,dx=-\sum_j\int_{a_j}^{b_j}e^{i\theta_j}g\d_x\eta \,dx.
\end{align*}
On each interval $(a_j,b_j)$ by integration by parts and $g(a_j)=g(b_j)=0$, we obtain
\[
\int_\Omega\eta\d_x (\Theta g)\,dx=\sum_j\int_{a_j}^{b_j}e^{i\theta_j}\eta \d_x g \,dx.
\]
Furthermore the vanishing derivative Lemma \ref{lemma:LL} implies $\d_x g=0$ a.e. outside $\{x;g(x)\ne0\}=\underset{j}{\cup}(a_j,b_j)$, therefore we can conclude
\[
\int_\Omega\eta\d_x (\Theta g)\,dx=\int_\Omega\eta \Theta\d_xg\,dx,
\]
namely $\d_x (\Theta g)=\Theta\d_xg\in L^2(\Omega)$, which finishes the proof of the Lemma.
\end{proof}

\begin{rem}\label{rmk:non_uniq}
As pointed out before, a direct consequence of Lemma \ref{lemma:H1ext} is the non-uniqueness of wave function lifting for  given hydrodynamic data at $H^1$ level, due to the arbitrary phase shifts allowed on the components of the non-vacuum regions. More precisely, let $\psi\in H^1(\R)$ be a wave function associated to $(\sqrt\rho,\Lambda)$ in the sense of the Definition \ref{def:ass}, then for any piecewise phase shift function $\Theta$ of the form \eqref{eq:theta}, by formula \eqref{eq:dtheta} it is straightforward to check that $\Theta\psi\in H^1(\R)$ is another wave function associated to the same hydrodynamic data, whose polar factors are $P(\Theta\psi)=\Theta P(\psi)$.
\end{rem}

\begin{prop}\label{prop:psiconv}
Let us consider a sequence of hydrodynamic data $\{(\sqrt{\rho_n},\Lambda_n)\}$ which converges to $(\sqrt{\rho_0},\Lambda_0)$ in $H^1(\R)\times L^2(\R)$. Assume $\{\psi_n\}$ and $\psi_0$ are wave functions associated to $\{(\sqrt{\rho_n},\Lambda_n)\}$ and $(\sqrt{\rho_0},\Lambda_0)$, respectively.

Then there exists a subsequence $\psi_{n_k}$ and a piecewise constant phase shift $\Theta$ given by the formula \eqref{eq:theta} on the connected components of $\{\psi_0\ne0\}$, such that 
\[
\lim_{n_k\to\infty}\|\psi_{n_k}-\Theta\,\psi_0\|_{H^1(\R)}=0.
\]
\end{prop}

\begin{proof}
For the case that $\rho_0\equiv 0$, the Proposition is trivial. Therefore we assume $\rho_0$ is not identically $0$. 

By Lemma \ref{lemma:H1ext}, for any phase shift $\Theta$ given by \eqref{eq:theta}, we have $\Theta\psi_0\in H^1(\R)$ with
\[
\d_x(\Theta\psi_0)=\Theta\d_x\psi_0.
\] 
By the polar decomposition Lemma \ref{lemma:polar} we have 
\begin{equation}\label{eq:polar_n}
\psi_{n}=\sqrt{\rho_{n}}\phi_{n}\quad and\quad \d_x\psi_{n}=(\d_x\sqrt{\rho_{n}}+i\Lambda_{n})\phi_{n}
\end{equation}
with a polar factor $\phi_{n}\in P(\psi_{n})$. Since $\|\phi_n\|_{L^\infty}\leq1$, then up to a subsequence $\{\phi_{n_k}\}$ converges weakly* to some $\tilde\phi_0$ in $L^\infty(\R)$. We define $\tilde\psi_0=\sqrt{\rho_0}\tilde\phi_0$. Since $(\sqrt{\rho_n},\Lambda_n)$ converges strongly to $(\sqrt{\rho_0},\Lambda_0)$ in $H^1(\R)\times L^2(\R)$, by passing to the limit as $n_k\to\infty$ in \eqref{eq:polar_n} we obtain
\begin{align*}
\psi_{n_k}=\sqrt{\rho_{n_k}}\phi_{n_k}\rightharpoonup \sqrt{\rho_0}\tilde\phi_0=\tilde\psi_0,\quad \textrm{in}\;L^2(\R),
\end{align*}
and
\begin{align*}
\d_x\psi_{n_k}=(\d_x\sqrt{\rho_{n_k}}+i\Lambda_{n_k})\phi_{n_k}\rightharpoonup (\d_x\sqrt{\rho_0}+i\Lambda_0)\tilde\phi_0, \quad \textrm{in}\;L^2(\R).
\end{align*}
Consequently, $\tilde\psi_0$ is a weak limit of $\psi_{n_k}$, with $\tilde\psi_0\in H^1(\R)$ and
\[
\d_x\tilde\psi_0=(\d_x\sqrt{\rho_0}+i\Lambda_0)\tilde\phi_0.
\]
Moreover the polar factorization of $\psi_{n_k}$ and $\tilde\psi_0$ implies 
\begin{align*}
\|\psi_{n_k}\|_{H^1(\R)}^2=\|\sqrt{\rho_{n_k}}\|_{H^1(\R)}^2+\|\Lambda_{n_k}\|_{L^2(\R)}^2\to \|\sqrt{\rho_0}\|_{H^1(\R)}^2+\|\Lambda_0\|_{L^2(\R)}^2=\|\tilde\psi_0\|_{H^1(\R)}^2,
\end{align*}
then by the weak convergence and the convergence of norms we obtain $\psi_{n_k}\to\tilde\psi_0$ in $H^1(\R)$. Since $\tilde\psi_{0}$ is a wave function associated to $(\sqrt{\rho_0},\Lambda_0)$, by Lemma \ref{lemma:rot} on each component $(a_j,b_j)$ of $\{\psi_0\ne0\}$, there exist a unitary $e^{i\theta_j}$ such that
\[
\tilde\psi_0=e^{i\theta_j}\psi_0\;on\;(a_j,b_j).
\]
Let 
\[
\Theta=e^{i\sum_j\theta_j\mathbf{1}_{(a_j,b_j)}},
\]
then we have $\tilde\psi_0=\Theta\,\psi_0$ on $\R$.
\end{proof}

Now we are going to prove the wave function lifting at the $H^2$ level. Sobolev regularity for the wave function does not correspond to further regularity for the hydrodynamical quantities. For example if we consider $\psi(x)=x$ on $[-1, 1]$, then $\psi\in\mathcal C^\infty$, but on the other hand $\sqrt{\rho}$ is only Lipschitz, with $\d_{x}^2\sqrt{\rho}=\delta_0$. In fact, considering a wave function $\psi\in H^2$ will yield further information on the energy density and on $\lambda$. Another important fact we should emphasise when considering a wave function lifting $\psi\in H^2(\R)$ is that $\d_x\psi$ should be a continuous function. As mentioned before, $\psi$ should have well designed phase shifts on all the connected components of the set $\{\rho>0\}$, such that $\d_x\psi$ has no jump discontinuity at vacuum boundaries.

\begin{prop}\label{prop:lift2}
Let $(\sqrt\rho,\Lambda)$ be a hydrodynamic state. There exists a wave function $\psi\in H^2(\R)$, associated to $(\sqrt\rho,\Lambda)$ in the sense of Definition \ref{def:ass}, if and only if
\begin{itemize}
\item $(\sqrt{\rho}, \Lambda)$ is a GCP state in the sense of Definition \ref{def:GCP};
\item the energy density $e$ given in \eqref{eq:en_dens} is continuous.
\end{itemize}
Moreover, if we assume 
\begin{equation}\label{eq:bd_lift2}
\begin{aligned}
\|\sqrt{\rho}\|_{H^1}+\|\Lambda\|_{L^2}\leq& M_1,\\
\|\mathbf{1}_{\{\rho>0\}}\d_x J/\sqrt{\rho}\|_{L^2(\R)}+\|\lambda\|_{L^2(\R)}\leq &M_2,
\end{aligned}
\end{equation}
for some constants $M_1,M_2<\infty$, where $\lambda$ is defined in \eqref{eq:def_lambda}, then
it follows that
\begin{equation}\label{eq:H2_lift}
\|\psi\|_{H^2(\R)}\leq C(M_1,M_2).
\end{equation}
\end{prop}

In order to prove Proposition \ref{prop:lift2}, we first state the following Lemma, which shows the connection between the hydrodynamic quantities $\lambda$, $\frac{\d_xJ}{\sqrt\rho}\mathbf{1}_{\{\rho>0\}}$ and the $H^2$ regularity of the associated wave function $\psi$.

\begin{lem}\label{lemma:H2}
\begin{itemize}
\item[(1)] If the hydrodynamic state $(\sqrt\rho,\Lambda)$ is given through an associated wave function $\psi\in H^2(\R)$, then $(\sqrt{\rho}, \Lambda)$ is a GCP state and the identities
\begin{equation}\label{eq:201}
\frac{\d_x J}{\sqrt\rho}\mathbf{1}_{\{\rho>0\}}=\IM(\bar\phi\d_x^2\psi),\quad \lambda-f'(\rho)\sqrt\rho=-\frac12\RE(\bar\phi\d_x^2\psi),
\end{equation}
hold true in $L^2(\R)$, where $\phi\in P(\psi)$ is a polar factor of $\psi$.

\item[(2)] Let us assume $(\sqrt\rho,\Lambda)$ is a GCP state and let $\psi\in H^1(\R)$ be a wave function associated to $(\sqrt\rho,\Lambda)$, then $\psi\in H^2(a_j,b_j)$, for all connected components $(a_j,b_j)$ of the set $\{\rho>0\}$ as given in \eqref{eq:poscomp}, and we have 
\begin{equation}\label{eq:dx2psi}
\d_x^2\psi=\left[2(f'(\rho)\sqrt\rho-\lambda)+i\frac{\d_xJ}{\sqrt\rho}\right]\phi\quad \textrm{in}\; L^2(a_j,b_j).
\end{equation}
\end{itemize}
\end{lem}

\begin{proof}
Let us consider $(\sqrt\rho,\Lambda)$ given through an associated wave function $\psi\in H^2(\R)$, then it follows Proposition \ref{prop:lift1} that $(\sqrt\rho,\Lambda)\in H^1(\R)\times L^2(\R)$. Moreover, by the definition of the hydrodynamic quantities $\rho=|\psi|^2$, $J=\IM(\bar\psi\d_x\psi)$, we have on the set $\{\rho>0\}=\{|\psi|>0\}$ 
\[
\IM(\bar\phi\d_x^2\psi)=\frac{\d_x \IM(\bar\psi\d_x\psi)}{|\psi|}=\frac{\d_x J}{\sqrt\rho}.
\]
Moreover, by using the definition \eqref{eq:def_lambda} for $\lambda$, on the set $\{|\psi|>0\}$ we have
\begin{align*}
-\frac12\RE(\bar\phi\d_x^2\psi)=&-\frac{1}{2|\psi|}\left[\d_x\RE(\bar\psi\d_x\psi)-|\d_x\psi|^2\right]\\
=&-\frac{1}{2\sqrt\rho}\left[\frac12\d_x^2\rho-(\d_x\sqrt\rho)^2-\Lambda^2\right]\\
=&-\frac12\d_x^2\sqrt\rho+\frac{\Lambda^2}{2\sqrt\rho}=\lambda-f'(\rho)\sqrt\rho.
\end{align*}
Let us emphasize that the previous computations are rigorous since $\psi\in H^2(\R)$ and $\{|\psi|>0\}$ is an open set.
By Lemma \ref{lemma:LL} we have $\d_x^2\psi=0$ a.e. $x$ in $\{\rho=0\}$, thus identities \eqref{eq:201} hold almost everywhere on $\R$.

Now let us assume $(\sqrt\rho,\Lambda)$ to be a GCP state, and let $\psi\in H^1(\R)$ be a wave function associated to $(\sqrt\rho,\Lambda)$. By using the polar factorization, it follows that
\[
\psi=\sqrt\rho\phi,\;\d_x\psi=(\d_x\sqrt\rho+i\Lambda)\phi,
\]
where $\phi\in P(\psi)$. Since $|\psi|>0$ on $(a_j,b_j)$, the polar factor $\phi\in P(\psi)$ is uniquely defined, $\phi=\frac{\psi}{|\psi|}$, and a direct computation gives
\begin{align*}
\d_x\phi=\frac{1}{|\psi|}\left(\d_x\psi-\d_x|\psi|\frac{\psi}{|\psi|}\right)
=\frac{i\Lambda}{\sqrt\rho}\phi.
\end{align*}
Using the identities above and the definition of hydrodynamic quantities, we obtain
\begin{align*}
\left[2(f'(\rho)\sqrt\rho-\lambda)+i\frac{\d_xJ}{\sqrt\rho}\right]\phi=&\left(\d_x^2\sqrt\rho+i\d_x\Lambda+i\d_x\sqrt\rho\frac{\Lambda}{\sqrt\rho}-\frac{\Lambda^2}{\sqrt\rho}\right)\phi\\
=&\d_x\left[(\d_x\sqrt\rho+i\Lambda)\phi\right]=\d_x(\d_x\psi).
\end{align*}
Thus we prove the identity \eqref{eq:dx2psi}, and by using the bounds of $\lambda$ and $\frac{\d_xJ}{\sqrt{\rho}}$ given in the Definition \ref{def:GCP} of GCP state, it follows that $\psi\in H^2(a_j,b_j)$.
\end{proof}
The previous Lemma shows that the condition
\begin{equation*}
\|\mathbf{1}_{\{\rho>0\}}\d_x J/\sqrt{\rho}\|_{L^2(\R)}+\|\lambda\|_{L^2(\R)}\leq M_2,
\end{equation*}
on the hydrodynamic state allows to improve on the regularity of the wave function $\psi$ on the connected components of $\{\rho>0\}$.
However, in general $\psi$ is not in $H^2(\R)$, since $\d_x\psi$ may possibly experience jump discontinuities at vacuum boundaries, due to the fact that the $H^1$ regularity allows for arbitrary phase shifts on connected components, see the Remark \ref{rmk:non_uniq}.
On the other hand, $H^2$ regularity requires compatibility conditions of the phase shift at vacuum boundaries.
We will show that, by additionally assuming the continuity of energy density, it is possible to provide an appropriate choice of the phase shifts on every connected component, which will enable us to construct another wave function $\tilde\psi\in H^2$ associated to the same hydrodynamical quantities. More precisely, starting from $\psi$, we will construct a $\tilde\psi\in H^2(\R)$ such that $\tilde\psi=\Theta\psi$, where $\Theta$ is a piecewise phase shift of the form
\begin{equation}
\Theta=\exp\left(i\sum_j\theta_j\mathbf{1}_{(a_j,b_j)}\right),\quad \theta_j\in[0,2\pi).
\end{equation}
The main difficulty of our construction lies in the accumulation points of vacuum boundaries. To overcome this difficulty, we need the next Lemma which shows that under the assumption of Proposition \ref{prop:lift2}, the energy density $e$ vanishes at the accumulation points of vacuum boundaries, and $\sqrt{e}$ is indeed a function in the space $H^1(\R)$.

\begin{lem}\label{lemma:H1e}
Let $(\sqrt\rho,\Lambda)$ be a GCP hydrodynamic state and let us assume that the energy density $e$ is continuous. If $\tilde x$ is an accumulation point of vacuum boundaries $\underset{j}{\cup}\{a_j,b_j\}$, then $e(\tilde x)=0$.
\end{lem}

\begin{proof}
By using Proposition \ref{prop:lift1}, there exists a $\psi\in H^1(\R)$ associated to $(\sqrt\rho,\Lambda)$, and the energy density is given by
\[
e=\frac12|\d_x\psi|^2+f(|\psi|^2).
\]
In particular, the continuity of $e$ implies that $|\d_x\psi|$ is also a continuous function. To prove Lemma \ref{lemma:H1e}, it is sufficient to show $|\d_x\psi|$ vanishes at the accumulation points of vacuum boundaries.

Let us consider  an accumulation point $\tilde x$ of the vacuum boundaries $\cup_j\{a_j,b_j\}$, namely there exists a sequence of intervals $\{(a_{j_k},b_{j_k})\}_k$ such that $(a_{j_k},b_{j_k})\to \tilde x$ as $k\to\infty$. On each $(a_{j_k},b_{j_k})$, we take $x_{j_k}=\frac12(a_{j,k}+b_{j_k})$, and we claim $\d_x\psi(x_{j_k})\to 0$ as $k\to\infty$. If the claim holds, by the continuity of $|\d_x\psi|$ we obtain $|\d_x\psi(\tilde x)|=0$. 

The claim $\d_x\psi(x_{j_k})\to 0$ can be proved by a contradiction. Indeed, we assume that
$$\underset{k}{\liminf}|\d_x\psi(x_{j_k})|>c>0.$$
Since $a_{j_k}$ and $b_{j_k}$ are vacuum boundary points, namely $\psi(a_{j_k})=\psi(b_{j_k})=0$, then we have
\[
0=\int_{a_{j_k}}^{b_{j_k}}\d_x\psi(y)dy.
\]
Moreover by the property (2) of Lemma \ref{lemma:H2}, we have $\psi\in H^2(a_{j_k},b_{j_k})$, so we can decompose the integral as
\begin{align*}
0=&\int_{a_{j_k}}^{b_{j_k}}\left[\d_x\psi(x_{j_k})+\int_{x_{j_k}}^{y}\d_x^2\psi(s)ds\right]dy\\
=&\d_x\psi(x_{j_k})\delta_{j_k}+\int_{a_{j_k}}^{b_{j_k}}\int_{x_{j_k}}^{y}\d_x^2\psi(s)dsdy,
\end{align*}
where $\delta_{j_k}=b_{j_k}-a_{j_k}$. By using the identity \eqref{eq:dx2psi} and the bounds \eqref{eq:bd_lift2}, it follows that
\[
\|\d_x^2\psi\|_{L^2(a_{j_k},b_{j_k})}\leq 2\|f'(\rho)\sqrt\rho-\lambda\|_{L^2}+\|\frac{\d_xJ}{\sqrt\rho}\mathbf{1}_{\{\rho>0\}}\|_{L^2}\leq C(M_1,M_2),
\]
then we have
\[
\left|\int_{a_{j_k}}^{b_{j_k}}\int_{x_{j_k}}^{y}\d_x^2\psi(s)dsdy\right|\leq \delta_{j_k}^\frac32\|\d_x^2\psi\|_{L^2(a_{j_k},b_{j_k})}
\leq \delta_{j_k}^\frac32 C(M_1,M_2).
\]
For $k$ large enough one has $|\d_x\psi(x_{j_k})|>c$ and $\delta_{j_k}^\frac12 C(M_1,M_2)<\frac{c}{2}$ since $\delta_{j_k}\to0$ as $k\to\infty$, hence it follows that
\begin{align*}
0\ge |\d_x\psi(x_{j_k})\delta_{j_k}|-\left|\int_{a_{j_k}}^{b_{j_k}}\int_{x_{j_k}}^{y}\d_x^2\psi(s)dsdy\right|>\frac{c}{2}\delta_{j_k}>0,
\end{align*}
which is a contradiction. 
\end{proof}

Now we are ready to prove Proposition \ref{prop:lift2}.

\begin{proof}[Proof of Proposition \ref{prop:lift2}]
Let us consider the hydrodynamic data $(\sqrt\rho,\Lambda)$ given through an associated wave function $\psi\in H^2(\R)$, then by the property (1) of Lemma \ref{lemma:H2} we know that $(\sqrt\rho,\Lambda)$ is a GCP state.
Moreover, by the definition \eqref{eq:en_dens} of the energy density and by the polar factorization \eqref{eq:quad}
\[
e=\frac12(\d_x\sqrt\rho)^2+\frac12\Lambda^2+f(\rho)=\frac12|\d_x\psi|^2+f(|\psi|^2),
\]
and therefore $e$ is a continuous function.

Now let us assume that $(\sqrt\rho,\Lambda)$ is a GCP state, and  decompose the non vacuum regions into disjoint intervals,
\begin{equation}\label{eq:non_vac}
\{\rho>0\}=\cup_{j}(a_j, b_j),\quad \rho(a_j)=\rho(b_j)=0.
\end{equation}
Lemma \ref{lemma:H2} shows that on all the connected components $(a_j,b_j)$ of $\{\rho>0\}$ we have $\psi\in H^2(a_j,b_j)$, and the identity \eqref{eq:dx2psi} holds true in $L^2(a_j,b_j)$. However as discussed before, to obtain a $H^2$ wave function defined on the whole $\R$, we need to overcome the possible jump discontinuity at vacuum boundaries. The additionally assumption on the continuity of energy density allows us to provide an well-designed choice of the phase shifts on every connected component and to construct another wave function $\tilde\psi\in H^2$ associated to the same hydrodynamical quantities. More precisely, starting from $\psi$, we will construct a piecewise phase shift function $\Theta$ of the form
\begin{equation}\label{eq:sigmarot}
\Theta=\exp\left(i\sum_j\theta_j\mathbf{1}_{(a_j,b_j)}\right),\quad \theta_j\in[0,2\pi),
\end{equation} 
such that $\tilde\psi=\Theta\psi$ belongs to $H^2(\R)$.

Before the construction of $\Theta$, we first give some discussion on vacuum boundaries $\cup_j\{a_j,b_j\}$. Let $\{\mathring{a}_j\}$ denote the isolated vacuum points, namely $\rho(\mathring{a}_j)=0$ and $\rho>0$ on $(\mathring{a}_j-\epsilon,\mathring{a}_j+\epsilon)\setminus\{\mathring{a}_j\}$ for some small $\epsilon>0$. Then we define the set
\[
W=\{\rho>0\}\underset{j}{\cup}\{\mathring{a}_j\}.
\]
By the continuity of $\rho$ and the definition of $\mathring{a}_j$, it is straightforward to see that $W$ is an open set, consequently it can be represented as disjoint open intervals
\[
W=\underset{j}{\cup}(\bar{a}_j,\bar{b}_j),
\]
where $\bar{a}_j,\bar{b}_j\in\cup_j\{a_j,b_j\}\setminus\cup_j\{\mathring{a}_j\}$. The set $\cup_j\{a_j,b_j\}\setminus\cup_j\{\mathring{a}_j\}$ consists of the following two types of vacuum boundary points. The first type is the boundary of large vacuum with positive measure, namely $|\psi|^2=\rho \equiv 0$ on $(\bar{a}_j-\epsilon,\bar{a}_j)$ or on $(\bar{b}_j,\bar{b}_j+\epsilon)$ for some $\eps>0$, and by the continuity of $|\d_x\psi|$ we have $|\d_x\psi(\bar{a}_j)|=|\d_x\psi(\bar{b}_j)|=0$. The second type is the accumulation point of vacuum boundaries, where $|\d_x\psi|$ also vanishes as proved in Lemma \ref{lemma:H1e}. 

Let us fix an interval $(\bar{a}_j,\bar{b}_j)$ and let $\{\mathring{a}_{j_k}\}_{k=K_1}^{K_2}$ denote the vacuum points in $(\bar{a}_j,\bar{b}_j)$. By our construction, the only possible accumulation points of $\{\mathring{a}_{j_k}\}$ are $\bar{a}_j$ and $\bar{b}_j$, hence we can assume $\mathring{a}_{j_k}<\mathring{a}_{j_{k+1}}$ for all $k$. 

The phase shift function $\Theta$ on $(\bar{a}_j,\bar{b}_j)$ is constructed through the following strategy. We start from a fixed $(\mathring{a}_{j_0},\mathring{a}_{j_1})$ and set $\theta_{j_0}=0$, namely $\Theta=1$ on $(\mathring{a}_{j_0},\mathring{a}_{j_1})$. To extend the definition of $\Theta$ to $(\mathring{a}_{j_1},\mathring{a}_{j_2})$, we notice that the intervals $(\mathring{a}_{j_k},\mathring{a}_{j_{k+1}})$ are connected components of the set $\{\rho>0\}$. Thus by Lemma \ref{lemma:H2} we have the $H^2$ regularity of $\psi$ on $(\mathring{a}_{j_0},\mathring{a}_{j_1})$ and $(\mathring{a}_{j_1},\mathring{a}_{j_2})$, which implies the the left and right side limits $\d_x\psi(\mathring{a}_{j_1}^-)$ and $\d_x\psi(\mathring{a}_{j_1}^+)$ exist. On the other hand, the continuity of the energy density $e$ implies $|\d_x\psi|=\sqrt{2e-f(\rho)}$ is continuous, hence we have $|\d_x\psi(\mathring{a}_{j_1}^-)|=|\d_x\psi(\mathring{a}_{j_1}^+)|$. Therefore we can choose a $\theta_{j_1}\in [0,2\pi)$ such that $(\Theta\cdot\d_x\psi)(\mathring{a}_{j_1}^-)=e^{i\theta_{j_1}}\d_x\psi(\mathring{a}_{j_1}^+)$ ($\theta_{j_1}=0$ if $|\d_x\psi|$ vanishes at $\mathring{a}_{j_1}$), and we define $\Theta=e^{i\theta_{j_1}}$ on $(\mathring{a}_{j_1},\mathring{a}_{j_2})$. By repeating this process inductively, we extend the definition of $\Theta$ to the whole $(\bar{a}_j,\bar{b}_j)$. Moreover we have $\Theta\d_x\psi\in H^1(\mathring{a}_{j_k},\mathring{a}_{j_k})$ and our construction ensures the continuity of $\Theta\d_x\psi$ at all point vacuum $\mathring{a}_{j_k}$. Thus by the elementary property of Sobolev functions, it follows that $\Theta\d_x\psi\in H^1(\bar{a}_j,\bar{b}_j)$, and by Lemma \ref{lemma:H1ext} and Lemma \ref{lemma:H2} we have the identity
\begin{equation}\label{eq:401}
\d_x(\Theta\d_x\psi)=\left[\d^2_x\sqrt\rho-\frac{\Lambda^2}{\sqrt\rho}+i\frac{\d_xJ}{\sqrt\rho}\right]\Theta\phi\quad in\;L^2(\bar{a}_j,\bar{b}_j),\quad \phi\in P(\psi).
\end{equation}

To this point we have constructed the phase shift function $\Theta$ on $W=\cup_j(\bar{a}_j,\bar{b}_j)$, and we extend $\Theta$ to the whole $\R$ by definition $\Theta=1$ on $W^c$, thus $\Theta$ has the form \eqref{eq:sigmarot} as required, and we define $\tilde\psi=\Theta\psi$. By Lemma \ref{lemma:H1ext} we have $\d_x\tilde\psi=\Theta\d_x\psi$, it is straightforward to check $\tilde\psi\in H^1(\R)$ is also a wave function associated to $(\sqrt\rho,\Lambda)$ with the polar factorization 
\begin{equation}
\d_x\tilde\psi=(\d_x\sqrt\rho+i\Lambda)\Theta\phi.
\end{equation}
Now we claim $\d_x\tilde\psi=\Theta\d_x\psi\in H^1(\R)$. Let $\eta\in C_c^\infty(\R)$ be an arbitrary test function, then we have
\[
\int_{\R}\eta\,\d_x(\Theta\d_x\psi)dr\coloneqq-\int_{\R}(\Theta\d_x\psi)\d_x\eta \,dr.
\]
By the vanishing derivative Lemma it follows that $\d_x\psi=0$ a.e. outside the set $W$, therefore
\[
\int_{\R}\eta\,\d_x(\Theta\d_x\psi)dr=-\int_{W}(\Theta\d_x\psi)\d_x\eta\,dx=-\underset{j}{\sum}\int_{\bar{a}_j}^{\bar{b}_j}(\Theta\d_x\psi)\d_x\eta\,dx.
\]
Since $\Theta\d_x\psi\in H^1(\bar{a}_j,\bar{b}_j)$, we can write
\[
\int_{\bar{a}_j}^{\bar{b}_j}(\Theta\d_x\psi)\d_x\eta\,dx=-\int_{\bar{a}_j}^{\bar{b}_j}\eta\,\d_x(\Theta\d_r\psi)dx+\eta(\bar{b}_j)(\Theta\d_x\psi)(\bar{b}_j)-\eta(\bar{a}_j)(\Theta\d_x\psi)(\bar{a}_j),
\]
where the boundary terms vanish since $\d_x\psi(\bar{a}_j)=\d_x\psi(\bar{b}_j)=0$ as we have shown before. By using the identity \eqref{eq:401} and the bounds for GCP states in Definition \ref{def:GCP}, we obtain
\begin{align*}
\left|\int_{\R}\eta\,\d_x(\Theta\d_x\psi)dx\right|\leq & \|\eta\|_{L^2(\R)}\left(\|\d^2_x\sqrt\rho-\frac{\Lambda^2}{\sqrt\rho}\|_{L^2(W)}+\|\frac{\d_rJ}{\sqrt\rho}\|_{L^2(W)}\right)\\
\leq & C(M_1,M_2)\|\eta\|_{L^2(\R)}.
\end{align*}
Thus we conclude $\d_x\tilde\psi=\Theta\d_x\psi\in H^1(\R)$. Moreover again by the vanishing derivative Lemma, we have $\d_x(\Theta\d_x\psi)=0$ a.e. on the set $\{\rho=0\}=\{\psi=0\}$, therefore the identity
\begin{equation}\label{eq:402}
\d_x^2\tilde\psi=\d_x(\Theta\d_x\psi)=\left[\d^2_x\sqrt\rho-\frac{\Lambda^2}{\sqrt\rho}+i\frac{\d_xJ}{\sqrt\rho}\right]\mathbf{1}_{\{\rho>0\}}\Theta\phi
\end{equation}
holds true in $L^2(\R)$, and the bound \eqref{eq:H2_lift} for the wave function $\tilde\psi$ follows straightforwardly.
\end{proof}

We conclude this section by proving Theorems \ref{thm:glob} and \ref{thm:glob2}.

\begin{proof}[Proof of Theorems \ref{thm:glob} and \ref{thm:glob2}]
Let us consider $(\sqrt{\rho_0}, \Lambda_0)$ satisfying \eqref{eq:C1_intro} and such that $\Lambda_0=0$ a.e. on $\{\rho_0=0\}$. By Proposition \ref{prop:lift1} we know there exists a $\psi_0\in H^1$ associated to $(\sqrt{\rho_0}, \Lambda_0)$. Let $\psi\in\mathcal C([0,T]; H^1(\R))$ be the solution to \eqref{eq:NLS_1d} with initial datum $\psi(0)=\psi_0$ given by Theorem \ref{thm:NLS}, then the polar factorization method and Proposition \ref{prop:old_1d} imply that $\sqrt{\rho}=|\psi|$, $\Lambda=\IM(\bar\phi\d_x\psi)$, $\phi\in P(\psi)$, is a global in time finite energy weak solution to \eqref{eq:QHD_1d}. Furthermore by using also Proposition \ref{prop:old_1d} we know that the total energy is a conserved quantity.

Now we are going to prove Theorem \ref{thm:glob2} regarding the $H^2$ regularity. If $(\sqrt{\rho_0}, \Lambda_0)$ satisfies \eqref{eq:C1_intro},\eqref{eq:C2_intro} and the continuity of the energy density, then by Proposition \ref{prop:lift2} we can choose the initial wave function to be $\psi_0\in H^2(\R)$. 
By the persistence of regularity for the NLS equation \eqref{eq:NLS_1d}, see Theorem \ref{thm:NLS}, for any $0<T<\infty$  we have $\psi\in\mathcal C([0,T];H^2(\R))\cap\mathcal C^1([0,T];L^2(\R))$ and property \eqref{eq:H2} holds. By using the bound \eqref{eq:H2_lift}, property \eqref{eq:H2} implies
\begin{equation}\label{eq:dt}
\|\d_t\psi\|_{L^\infty_tL^2_x}\leq C(T, M_1, M_2).
\end{equation}
Now, let us recall the definition of $\lambda=-\IM(\bar\phi\d_t\psi)$, with $\phi\in P(\psi)$, then by the polar factorization and \eqref{eq:dt}, we get
\begin{equation}\label{eq:higher_unif}
\|\d_t\sqrt{\rho}\|_{L^\infty_tL^2_x}+\|\lambda\|_{L^\infty_tL^2_x}\leq C(T, M_1, M_2),
\end{equation}
which implies that the functional given in \eqref{eq:higher} is uniformly bounded for a.e. $t\in[0, T]$.
Moreover,
\begin{equation*}
\sqrt{\rho}\lambda=-\IM(\bar\psi\d_t\psi)=-\frac14\d_x^2\rho+\frac12(\d_x\sqrt{\rho})^2+\frac12\Lambda^2+\rho f'(\rho),
\end{equation*}
so that \eqref{eq:lambda_intro} holds.
Next we are going to show the bounds in \eqref{eq:B2_c}. By the conservation of energy we know that
\begin{equation*}
\|\sqrt{\rho}\|_{L^\infty_tH^1_x}+\|\Lambda\|_{L^\infty_tL^2_x}\leq M_1.
\end{equation*}
Since $\d_x^2\rho=2\RE(\bar\psi\d_x^2\psi)+2|\d_x\psi|^2$, by standard H\"older's inequality, Sobolev embedding and \eqref{eq:H2}, we get
\begin{equation*}
\|\d_x^2\rho\|_{L^\infty_tL^2_x}\lesssim\|\psi\|^2_{L^\infty_tH^2_x}\leq C(T, M_1, M_2).
\end{equation*}
On the other hand, by the continuity equation, standard Sobolev embedding and \eqref{eq:higher_unif}, we have
\begin{equation*}
\|\d_xJ\|_{L^\infty_tL^2_x}=\|\d_t\rho\|_{L^\infty_tL^2_x}\lesssim\|\sqrt{\rho}\|_{L^\infty_{t, x}}\|\d_t\sqrt{\rho}\|_{L^\infty_tL^2_x}\leq C(T, M_1, M_2).
\end{equation*}
By the polar factorization we have 
\[
e=\frac12(\d_x\sqrt{\rho})^2+\frac12\Lambda^2+f(\rho)=\frac12|\d_x\psi|^2+f(|\psi|^2),
\]
therefore it follows that
\begin{equation*}
\|\d_x\sqrt{e}\|_{L^\infty_tL^2_x}\leq\|\psi\|_{L^\infty_tH^2_x}\leq C(T, M_1, M_2),
\end{equation*}
by using the elementary inequality $\d_x|\d_x\psi|\leq|\d_{x}^2\psi|$ a.e. in $\R$.
To conclude the proof of Theorem \ref{thm:glob2} we have to show the entropy equality \eqref{eq:en_diff}. Indeed we have (by using a standard approximation process)
\begin{align*}
\d_te=&\frac12\d_t|\d_x\psi|^2+\d_tf(\rho)\\
=&\RE(\d_x\bar\psi\d_x\d_t\psi)+f'(\rho)\d_t\rho\\
=&\d_x\RE(\d_x\bar\psi\d_t\psi)-\RE(\d_x^2\bar\psi\d_t\psi)+f'(\rho)\d_t\rho.
\end{align*}
By using the continuity equation for $\rho$ and the NLS equation, the second term on the right hand side of the previous equality can be written as
\begin{align*}
\RE(\d_x^2\bar\psi\d_t\psi)=&\RE[\d_x^2\bar\psi(\frac{i}{2}\d_x^2\psi-if'(\rho)\psi)]\\
=&-f'(\rho)\IM(\bar\psi\d_x^2\psi)=-f'(\rho)\d_x J=f'(\rho)\d_t\rho.
\end{align*}
On the other hand we can decompose $\RE(\d_x\bar\psi\d_t\psi)$ as
\[
\RE(\d_x\bar\psi\d_t\psi)=\RE(\bar\phi\d_x\psi)\RE(\bar\phi\d_t\psi)+\IM(\bar\phi\d_x\psi)\IM(\bar\phi\d_t\psi),
\]
where $\phi\in P(\psi)$ is a polar factor,then by the polar factorization and since $\d_t\sqrt\rho=\RE(\bar\phi\d_t\psi)$ and $\lambda=-\IM(\bar\phi\d_t\psi)$, we obtain the entropy equality
\[
\d_te+\d_x(\Lambda\lambda-\d_x\sqrt\rho\d_t\sqrt\rho)=0.
\]
\end{proof}
\begin{rem}\label{rem:gen_press}
The proof of Theorems \ref{thm:glob} and \ref{thm:glob2} heavily relies on the wave function lifting results developed in this section, in order to define an initial datum $\psi_0$, which then we let evolve according to equation \eqref{eq:NLS_1d}. Thus we exploit the well-posedness result for the nonlinear Schr\"odinger equation \eqref{eq:NLS_1d}, already established in the literature, and we use the polar factorization of Lemma \ref{lemma:polar} in order to define the quantum state $(\sqrt{\rho}, \Lambda)$ and consequently the weak solution $(\rho, J)$. It is straightforward to see that the same procedure can be applied as long as, instead of the nonlinearity considered in the equation \eqref{eq:NLS_1d}, we have any other nonlinearity for which a global well-posedness result still holds. In particular, this implies that it is possible to consider more general pressure laws, even by relaxing the convexity assumption on the internal energy, namely the monitonicity assumption on the pressure.
\end{rem}

\section{Dispersive estimates}\label{sect:disp}
In this section we will collect some a priori estimates satisfied by finite energy weak solutions to $\eqref{eq:QHD_1d}$. If we restricted our analysis to Schr\"odinger-generated solutions — like the ones constructed in the proof of Theorem \ref{thm:glob} — then the dispersive estimates inherited by the NLS dynamics would lead to a wide range of information. It is not the case, however, for general solutions, where the quasi-linear nature of the system \eqref{eq:QHD_1d} prevents to successfully exploit semigroup techniques to infer suitable smoothing estimates.

Consequently beyond the standard conservation laws we need a new class of intrinsically hydrodynamic estimates.
Some of these estimates however will be the natural counterpart of the already existing estimates for the nonlinear Schr\"odinger equations, but we will prove them even for hydrodynamic solutions which are not Schr\"odinger-generated. The first one of such estimates will be  the hydrodynamical analogue of the bounds inferred by the pseudo-conformal energy for NLS solutions \cite{GV}, see also \cite{HNT1, HNT2}. In the NLS case we may consider a functional related to the $L^2$ norm of $(x+it\nabla)\psi$ (due to the Galilean invariance of the NLS). By exploiting the polar factorization we have
\begin{equation*}
\frac12\|(x+it\nabla)\psi(t)\|_{L^2}^2
=\int\frac{|x|^2}{2}\rho(t, x)\,dx-t\int x\cdot J(t, x)\,dx+\frac{t^2}{2}\int|\Lambda|^2+|\nabla\sqrt{\rho}|^2\,dx.
\end{equation*}
The right hand side of this identity suggests to introduce the following functional (already used for the NLS equation in \cite{HNT1}),
\begin{equation}\label{eq:pc_en}
H(t)=\int_{\R}\frac{x^2}{2}\rho(t, x)\,dx-t\int_{\R}x\,J(t, x)\,dx+t^2E(t),
\end{equation}
where the energy $E(t)$ is defined in \eqref{eq:en_QHD}. In the next Proposition we show that, by assuming a control on the initial momentum of inertia, the functional $H(t)$ will be bounded for any finite energy weak solutions
\begin{prop}\label{prop:pseudo}
Let $(\rho, J)$ be a weak solution to \eqref{eq:QHD_1d} with finite mass and energy as in the Definition \ref{def:FEWS} and \ref{def:FEWS_2}, such that the total energy $E(t)$ is non-increasing, and
\begin{equation}\label{eq:fin_var}
\int_{\R}x^2\rho_0(x)\,dx<\infty.
\end{equation}
Then we have
\begin{equation*}
H(t)+\int_0^ts\int_{\R}\rho f'(\rho)-3f(\rho)\,dxds\leq\int_{\R}\frac{x^2}{2}\rho_0(x)\,dx.
\end{equation*}
\end{prop}
\begin{proof}
In order to prove the Proposition, we need to choose suitable test functions $\eta,\zeta\in C^\infty_c([0,T)\times\R)$ in the Definition 8 of weak solutions, which are obtained by introducing the following cut-off functions in space and time. For $R>0$, we defined the cut-off functions in space as
\[
\Phi_R(x)=\frac{x^2}{2}\chi_R^2(x),
\]
where $\chi_R(x)=\chi(x/R)\in C^\infty_c(\R)$, such that $\chi(x)\equiv 1$ for $|x|\leq 1$, and $\textrm{supp}\chi\subset\{|x|\leq 2\}$. On the other hand, let us consider convex functions $\{\tilde\chi_\eps\}_{\eps>0}\subset C^\infty(\R)$ such that
\[
\tilde\chi_\eps(s)=\begin{cases}
0 &,\ s<-\eps\\
s &,\ s>\eps
\end{cases}.
\]
it is clear to see that $\{\tilde\chi_\eps''\}$ is a sequence of mollifiers with $\|\tilde\chi_\eps''\|_{L^1}=1$, which approximates the Dirac-delta in the space of measure, and $\textrm{supp}\tilde\chi_\eps''\subset\{|s|\leq\eps\}$.

We first prove that the momentum of inertia increases at most quadratically in time under the initial assumption \eqref{eq:fin_var}, namely
\begin{equation}\label{eq:fin_var_2}
\int_\R x^2\rho(t,x)dx\leq C t^2+\int_\R x^2\rho_0(x)dx.
\end{equation}
For $\epsilon<t<T-\epsilon$, we choose $\eta(s,x)=\tilde\chi_\eps'(t-s)\Phi_R(x)$ in \eqref{eq:QHD_cty}, then we have
\begin{equation}\label{eq:prop22_3}
-\int_0^T\int_{\R}\tilde\chi_\eps''(t-s)\Phi_R(x)\rho+\tilde\chi_\eps'(t-s)\Phi_R'(x)J\,dxds+\int_{\R}\tilde\chi_\eps'(t)\Phi_R(x)\rho_0(x)\,dx=0.
\end{equation}
By our choice of $\tilde\chi_\eps$, we have $\tilde\chi_\eps'(t)=1$ for $t>\eps$. Furthermore $\{\tilde\chi_\eps''\}$ is a sequence of  mollifiers, and $\tilde\chi_\eps\to \mathbf{1}_{\{s>0\}}s$ a.e. $s\in\R$. Therefore by the property of mollifiers and dominated convergence theorem, by letting $\epsilon\to 0$ in \eqref{eq:prop22_3} we get
\begin{equation}\label{eq:prop22_4}
-\int_{\R}\Phi_R(x)\rho(t,x)dx+\int_0^t\int_{\R}\Phi_R'(x)J\,dxds+\int_{\R}\Phi_R(x)\rho_0(x)\,dx=0,
\end{equation}
for a.e. $t\in(0,T)$. Let us define 
\[
V_R(t)=\int_\R \Phi_R(x)\rho(t,x)dx,
\] 
then by the definition of $\phi_R$ and the finite energy, we have
\begin{align*}
\int_{\R}\Phi_R'(x)J\,dx=&\int_{\R}(x\,\chi^2_R+x^2\chi_R\chi'_R)(\sqrt\rho\Lambda)\,dx\\
\leq & \|\chi_R+x\chi'_R\|_{L^\infty(\R)}\|\Lambda\|_{L^2(\R)}\left(\int_\R x^2\chi_R^2(x)\rho\,dx\right)^\frac12\\
\leq & C \left(\int_\R \Phi_R(x)\rho\,dx\right)^\frac12=C\, V_R^\frac12(t).
\end{align*}
By substituting the last inequality into \eqref{eq:prop22_4} and using Gronwall's argument, we get
\[
V_R(t)\leq C t^2+V_R(0),
\]
which will imply \eqref{eq:fin_var_2} as $R\to\infty$ . As a direct consequence of \eqref{eq:fin_var_2}, for a.e. $0<t<T$ we have 
\[
\int_\R |x\,J(t,x)|dx\leq \|\Lambda(t)\|_{L^2(\R)}(\left(\int_\R x^2\rho(t)dx\right)^\frac12<\infty,
\]
and by the dominated convergence theorem
\begin{equation}\label{eq:conv_xJ}
\int_0^t\int_{\R}\Phi_R'(x)J\,dxds\to\int_0^t\int_{\R}x\,J\,dxds.
\end{equation}

Next we set $\zeta_1(s,x)=\tilde\chi_\eps(t-s)\phi_R'(x)$ and $\zeta_2(s,x)=\tilde\chi'_\eps(t-s)\phi_R'(x)$ in
 \eqref{eq:QHD_mom}. As before, our choice of $\tilde\chi_\eps$ allow us to let $\eps\to 0$, then the weak formulation of the momentum equation \eqref{eq:QHD_mom} implies
\begin{equation}\label{eq:prop22_6}
\begin{aligned}
-\int_0^t\int_{\R}\phi_R'(x)J+&(t-s)\phi_R''(x)(\Lambda^2+p(\rho)+(\d_x\sqrt{\rho})^2)\\
&-\frac{1}{4}(t-s)\phi_R^{(4)}(x)\rho\,dxds+\int_{\R}t\,\phi_R'(x)J_0(x)\,dx=0,
\end{aligned}
\end{equation}
and
\begin{equation}\label{eq:prop22_8}
\begin{aligned}
-\int_{\R}\phi_R'(x)J(t,x)dx+&\int_0^t\int_{\R}\phi_R''(x)(\Lambda^2+p(\rho)+(\d_x\sqrt{\rho})^2)\\
&-\frac{1}{4}\phi_R^{(4)}(x)\rho(t,x)\,dxds+\int_{\R}\phi_R'(x)J_0(x)\,dx=0.
\end{aligned}
\end{equation}

Now we compute $t\times\eqref{eq:prop22_8}-\eqref{eq:prop22_6}-\eqref{eq:prop22_4}$,
\begin{equation}\label{eq:prop22_9}
\begin{aligned}
\int_{\R}\phi_R(x)\rho(t,x)dx-t\int_{\R}\phi_R'(x)J(t,x)dx&+\int^t_0\int_\R s\,\phi_R''(x)(\Lambda^2+p(\rho)+(\d_x\sqrt{\rho})^2)\\
&-\frac{s}{4}\phi_R^{(4)}(x)\rho(t,x)\,dxds-\int_{\R}\phi_R(x)\rho_0(x)dx=0.
\end{aligned}
\end{equation}
Recalling the definition of $\phi_R$, we have
\[
\phi''_R(x)=\chi_R^2(x)+2x(\chi_R^2)'+\frac{x^2}{2}(\chi_R^2)'',
\]
and it is straightforward to check $\|\phi''_R(x)\|_{L^\infty(\R)}\leq C$ and $\phi''_R(x)\to 1$ a.e. on $\R$ as $R\to\infty$. There by the  dominated convergence theorem we obtain
\begin{align*}
\int^t_0\int_\R s\,\phi_R''(x)(\Lambda^2+p(\rho)+(\d_x\sqrt{\rho})^2)dxds=&\int^t_0\int_\R s(\Lambda^2+p(\rho)+(\d_x\sqrt{\rho})^2)dxds\\
=&\int^t_0 2s\,E(s)+\int_\R(p(\rho)-2f(\rho))dxds,
\end{align*}
where $E(s)$ is the total energy given by \eqref{eq:en_QHD}. Also as $R\to\infty$, we have $\phi_R^{(4)}(x)\sim \mathcal{O}(R^{-2})$, hence by letting $R\to\infty$, the identity \eqref{eq:prop22_9} implies 
\begin{align*}
\int_{\R}\frac{x^2}{2}\rho(t,x)dx-t\int_{\R}x\,J(t,x)dx+&2\int^t_0 s\,E(s)dx\\
&=\int_{\R}\frac{x^2}{2}\rho_0(x)dx+2\int_0^t s\int_\R 2f(\rho)-p(\rho)dxds.
\end{align*}
Since $E(s)$ is non-increasing, we have
\[
2\int^t_0 s\,E(s)dx\geq t^2E(t),
\]
and we can conclude
\begin{align*}
H(t)= & \int_{\R}\frac{x^2}{2}\rho(t,x)dx-t\int_{\R}x\,J(t,x)dx+t^2E(t)\\
\leq & \int_{\R}\frac{x^2}{2}\rho_0(x)dx+2\int_0^t s\int_\R 3f(\rho)-\rho\,f'(\rho)dxds,
\end{align*}
where we also use $p'(\rho)=\rho\,f'(\rho)-f(\rho)$.
\end{proof}

In the case under our consideration we have $f(\rho)=\frac1\gamma\rho^\gamma$ so that
\begin{equation}\label{eq:pc}
H(t)+\left(1-\frac3\gamma\right)\int_0^ts\int\rho^\gamma\,dx\lesssim\int\frac{x^2}{2}\rho_0(x)\,dx.
\end{equation}
Furthermore, the functional in \eqref{eq:pc_en} can also be expressed as
\begin{equation}\label{eq:pc_en_2}
H(t)=\int_{\R}\frac{t^2}{2}(\d_x\sqrt{\rho})^2+\frac{t^2}{2}(\Lambda-\frac{x}{t}\sqrt{\rho})^2+t^2f(\rho)\,dx.
\end{equation}
A similar identity to \eqref{eq:pc} also appeared in Appendix A.2 in \cite{CDS}. In what follows we exploit \eqref{eq:pc}, \eqref{eq:pc_en_2} in order to infer a dispersive type estimate for the solutions to \eqref{eq:QHD_1d}, similarly to what is done in \cite{Bar} to study the scattering properties in the framework of NLS equations.

As we will explain later in Remark \ref{rmk:r_disp}, in the next Proposition we show the convergence of the velocity field towards a rarefaction wave and, in the case the total energy is conserved, asymptotically it will reduce to the total kinetic energy.
\begin{prop}\label{prop:disp}
Let $(\rho, J)$ be a weak solution to the system \eqref{eq:QHD_1d} satisfying the same assumptions as in Proposition \ref{prop:pseudo}, and the internal energy is given by $f(\rho)=\frac{1}{\gamma}\rho^\gamma$. Then we have
\begin{equation}\label{eq:disp_pc}
H(t)\lesssim t^{2(1-\sigma)}+\int\frac{x^2}{2}\rho_0(x)\,dx,
\end{equation}
where $\sigma=\min\{1, \frac{1}{2}(\gamma-1)\}$. In particular we have
\begin{equation}\label{eq:disp}
\|\d_x\sqrt{\rho}(t)\|_{L^2(\R)}\lesssim t^{-\sigma},
\end{equation}
\begin{equation}\label{eq:disp_kin}
\int(\Lambda(t, x)-\frac{x}{t}\sqrt{\rho}(t, x))^2\,dx\lesssim t^{-2\sigma}.
\end{equation}
and 
\begin{equation}\label{eq:disp_pr}
\int\rho^\gamma\,dx\lesssim t^{-2\sigma}.
\end{equation}
Moreover, if the total energy is conserved, then we have
\begin{equation}\label{eq:ld}
\lim_{t\to\infty}\frac12\|\Lambda(t)\|_{L^2}^2=E(0).
\end{equation}
\end{prop}

\begin{proof}
Let us consider the identity \eqref{eq:pc}, hence for $\gamma>3$ we have
\begin{equation*}
H(t)\leq\int\frac{x^2}{2}\rho_0(x)\,dx
\end{equation*}
and \eqref{eq:disp_pc} holds for $\sigma=1$.
Let us now consider the case $1<\gamma\leq3$, if we define
\begin{equation*}
F(t)=\frac{t^2}{\gamma}\int\rho^\gamma(t, x)\,dx
\end{equation*}
then by \eqref{eq:pc} and \eqref{eq:pc_en_2} we have
\begin{equation*}
F(t)\leq H(t)=(3-\gamma)\int_0^t\frac1s F(s)\,ds+\int_\R\frac{|x|^2}{2}\rho_0(x)\,dx.
\end{equation*}
By Gronwall we then have
\begin{equation*}
F(t)\lesssim t^{3-\gamma}F(1)+\int\frac{x^2}{2}\rho_0(x)\,dx,
\end{equation*}
which also implies 
\begin{equation*}
\int\rho^\gamma(t, x)\,dx\lesssim t^{1-\gamma}.
\end{equation*}
We can now plug the above estimate in \eqref{eq:pc} in order to obtain
\begin{equation*}
H(t)\lesssim\int_0^ts^{2-\gamma}\,ds+\int\frac{x^2}{2}\rho_0(x)\,dx,
\end{equation*}
which then implies \eqref{eq:disp_pc}.
\end{proof}

\begin{rem}\label{rmk:r_disp}
The dispersive estimates in Proposition \ref{prop:disp} have several consequences. First of all the mass density converges to zero and formally the velocity field $v$ asymptotically approaches a rarefaction wave, namely 
$v(t, x)\sim x/t$ as $t\to\infty$.
Moreover, as we have seen, asymptotically the motion becomes inertial and the total energy is given by the kinetic one. Intuitively this suggests that all the particles in the quantum gas tend to "synchronise" with a rarefaction wave in way that is reminiscent of the Landau damping \cite{MV}. Another important consequence which comes from a trivial application to the Sobolev embedding and the estimate \eqref{eq:disp} is given by
\begin{equation}\label{eq:disp_free}
\|\sqrt{\rho}(t)\|_{L^\infty}\lesssim t^{-\sigma/2}.
\end{equation}
When $\gamma\geq3$ this bound is analogous to the classical  dispersive estimate for the linear Schr\"odinger equation.
\end{rem}

Similar estimates appear in many contexts of evolution PDEs, see \cite{Straus, Gl, Sid, Mor}. In particular estimates like \eqref{eq:disp_pr} have been considered in classical compressible fluid dynamics \cite{Chem, Ser, Sid}, or the analogue of \eqref{eq:disp_kin} in kinetic theory, see for instance \cite{BD, Pe, IR}. These results are also somehow reminiscent of the vector field method \cite{Kl} (see also \cite{John}) used to study nonlinear wave equations, see also the recent paper \cite{FJS}.
We also mention \cite{Roz} where the author exploits the hydrodynamic formulation of nonlinear Schr\"odinger equation in order to study the formation of singularities for the mass-critical NLS. 
Moreover, for isothermal capillary fluids, namely when the pressure term is linear, or equivalently when we consider a Schr\"odinger equation with a logarithmic nonlinearity, the decay rate shows a logarithmic improvement, see \cite{CG, CCH}.
In the case of nonlinear Schr\"odinger equations, the fact that solutions to the nonlinear problem disperse as much as the linear solutions provides some additional information about their asymptotic behavior, see \cite{GV}. An alternative proof, based on the interaction Morawetz estimates \cite{Mor, LS}, has been recently developed to show the asymptotic completeness for the mass-supercritical, energy-subcritical NLS equations \cite{CKSTT, CGT, GVQ, PV}. 

The purpose of the next Proposition is to show that this kind of estimates actually hold also for arbitrary finite energy weak solutions to the QHD system \eqref{eq:QHD_1d}. 
Moreover, differently from the dispersive estimates in Proposition \ref{prop:disp} where it was necessary to assume a finite initial moment of inertia \eqref{eq:fin_var} and a non-increasing energy $E(t)$, in the next Proposition we only require the weak solutions to have finite energy.

The main idea of the next Proposition is inspired from the strategy of proof of the Morawetz estimates for nonlinear Schr\"odinger equations \cite{GVQ}, namely by using a convolution with a Riesz kernel associated to fractional powers of the inverse Laplacian. In the one dimensional case this analogy leads us to replace $(-\Delta)^{-1/2}$ with the anti-derivative.
\begin{thm}[Morawetz-type estimates]\label{thm:mor}
Let $(\rho, J)$ be a weak solution with finite mass and energy as in the Definition \ref{def:FEWS} and \ref{def:FEWS_2} satisfying
\[
M(t)+E(t)\leq M_1,
\]
then we have
\begin{equation}\label{eq:mor_hydro}
\|\d_x\rho\|_{L^2(\R_t\times\R_x)}^2+\|\rho\|_{L^{\gamma+1}(\R_t\times\R_x)}^{\gamma+1}\leq M_1^2.
\end{equation}
\end{thm}

\begin{proof}
Let us first define the anti-derivative of $J$ as 
\begin{equation}\label{eq:def_G}
G(t,x)=\int_{-\infty}^x J(t,s)ds.
\end{equation}
From our assumptions we have
\[
\|J\|_{L^\infty_tL^1_x}\leq \|\sqrt\rho\|_{L^\infty_tL^2_x}\|\Lambda\|_{L^\infty_tL^2_x}\leq M(t)^\frac12 E(t)^\frac12\leq M_1,
\]
hence it follows that $G\in L^\infty([0, T]\times\R)$ with the bound $\|G\|_{L^\infty_{t,x}}\leq M_1$. By integrating the equation for the momentum density $J$ in \eqref{eq:QHD_1d}, we get
\begin{equation}\label{eq:eqG}
\d_tG+\Lambda^2+(\d_x\sqrt\rho)^2+p(\rho)=\frac14\d_x^2\rho.
\end{equation}
Let us now consider the functional 
\begin{equation*}
\int_\R \rho(t,x)G(t,x)dx,
\end{equation*}
which is the analogue of the interaction Morawetz functional used for NLS equations. By differentiating with respect to time, we obtain 
\begin{align*}
\frac{d}{dt}\int_\R \rho(t,x)G(t,x)dx=&\int_\R G\,\d_t\rho\,dx+\int_\R \rho\,\d_tG\,dx\\
=&-\int_\R G\,\d_xJ\,dx-\int_\R \rho[\Lambda^2+(\d_x\sqrt\rho)^2+p(\rho)-\frac14\d_x^2\rho]dx\\
=&-\int_\R \rho\,p(\rho)dx-\frac12\int_\R(\d_x\rho)^2dx.
\end{align*}
By integrating in time the last identity, since $p(\rho)=(1-\frac1\gamma)\rho^\gamma$, we have
\begin{align*}
(1-\frac{1}{\gamma})\int_0^T\int_\R \rho^{\gamma+1}dxdt+\frac12\int_0^T\int_\R(\d_x\rho)^2dxdt=&\int_\R \rho(0,x)G(0,x)dx-\int_\R \rho(T,x)G(T,x)dx\\
\le&2\|\rho\|_{L^\infty_t L^1_x}\|G\|_{L^\infty_{t,x}}\le M_1^2.
\end{align*}
\end{proof}

\begin{rem}
The dispersive estimates stated in Proposition \ref{prop:disp} and Theorem \ref{thm:mor} hold, with suitable modifications, also in the multi-dimensional setting. Actually they are true even for solutions to Euler-Korteweg systems of the following type, see Proposition 35 in \cite{AMZ2},
\begin{equation}\label{eq:EK}
\left\{\begin{aligned}
&\d_t\rho+\diver J=0\\
&\d_tJ+\diver\left(\frac{J\otimes J}{\rho}\right)+\nabla p(\rho)=\rho\nabla\left(\diver(\kappa(\rho)\nabla\rho)-\frac12\kappa'(\rho)|\nabla\rho|^2\right).
\end{aligned}\right.
\end{equation}
By defining $K(\rho)=\int_0^\rho s\,\kappa(s)\,ds$, then the analogue of the Morawetz-type estimates in \eqref{eq:mor_hydro} for solutions to \eqref{eq:EK} is given by
\begin{equation*}
\||\nabla|^{\frac{3-d}{2}}\sqrt{\rho K(\rho)}\|_{L^2(\R_t\times\R^d_x)}^2
+\||\nabla|^{\frac{1-d}{2}}\sqrt{\rho p(\rho)}\|_{L^2(\R_t\times\R^d_x)}^2\lesssim M_1^2.
\end{equation*}
\end{rem}

\section{''Entropy and Morawetz type'' estimates}\label{sect:apri}
In this section we are going to discuss additional a priori estimates for solutions to \eqref{eq:QHD_1d}. In particular the estimates in  Proposition \ref{prop:unif0} will be exploited later to show the stability result of Theorem \ref{thm:stab}. As we will see, they are consequence of the bounds determining the GCP class given by the Definition \ref{def:cptsln} and the weak entropy inequality stated in the Definition \ref{def:entrsln}. The uniform bounds provided by the functional $I(t)$ defined in \eqref{eq:higher} will play a fundamental role. In particular in Proposition \ref{prop:higher} below we show that $I(t)$ can be controlled on compact time intervals. First of all in the case of smooth solutions with $\rho>0$, we will be able to show the validity of the conservation law \eqref{eq:en_diff} regarding the energy identity.
\begin{lem}\label{lem:en}
Let $(\rho, J)$ be a smooth solution to \eqref{eq:QHD_1d} such that $\rho>0$. Then the energy density $e$ satisfies the following conservation law
\begin{equation}\label{eq:en_cons}
\d_te+\d_x(\Lambda\lambda-\d_t\sqrt{\rho}\d_x\sqrt{\rho})=0.
\end{equation}
\end{lem}
\begin{proof}
Because of positivity and smoothness, we can rewrite the system  \eqref{eq:QHD_1d} as
\begin{equation*}
\left\{\begin{aligned}
&\d_t\rho+\d_x(\rho v)=0\\
&\rho\d_tv+\rho v\d_xv+\rho\d_xf'(\rho)
=\frac12\rho\d_x\left(\frac{\d_x^2\sqrt{\rho}}{\sqrt{\rho}}\right),
\end{aligned}\right.
\end{equation*}
therefore the equation for the momentum density can be equvalently written as
\begin{equation*}
\rho\d_tv+\rho\d_x\mu=0,
\end{equation*}
where $\mu$ denotes the chemical potential defined in \eqref{eq:chem}.
Similarly, the energy density satisfies
\begin{equation*}
e=\frac12(\d_x\sqrt{\rho})^2+\frac12\rho v^2+f(\rho),
\end{equation*}
and hence
\begin{equation*}
\begin{aligned}
\d_te=&\d_x\sqrt{\rho}\d_{tx}\sqrt{\rho}+\left(\frac12v^2+f'(\rho)\right)\d_t\rho+\rho v\d_tv\\
=&\d_x\left(\d_x\sqrt{\rho}\d_t\sqrt{\rho}\right)
+\left(-\frac12\frac{\d_x^2\sqrt{\rho}}{\sqrt{\rho}}+\frac12v^2+f'(\rho)\right)\d_t\rho+\rho v\d_tv.
\end{aligned}
\end{equation*}
By using the previous identities and the definition \eqref{eq:chem}, it follows that
\begin{equation*}\begin{aligned}
\d_te=&\d_x\left(\d_x\sqrt{\rho}\d_t\sqrt{\rho}\right)-\mu\d_x(\rho v)-\rho v\d_x\mu\\
=&\d_x\left(\d_x\sqrt{\rho}\d_t\sqrt{\rho}-\rho v\mu\right).
\end{aligned}
\end{equation*}
\end{proof}

\begin{rem}
The calculation in the previous Lemma requires smoothness and positivity of solutions in order to be rigorously justified; however it is an interesting question to see whether this conservation law, or its weaker version with the inequality \eqref{eq:entr_ineq}, holds for a larger class of solutions. We mention the paper \cite{TzavFasc} where the authors determine some conditions on the velocity field and the mass density which allow to show the conservation of energy.
Let us remark that also the solutions constructed in Theorem \ref{thm:glob2} satisfy the identity \eqref{eq:en_cons} in the distributional sense.
\end{rem}

By using the conservation law \eqref{eq:en_cons} we now show that the functional $I(t)$ is uniformly bounded on compact time intervals.

\begin{prop}\label{prop:higher}
Let $(\rho, J)$ be a smooth solution to \eqref{eq:QHD_1d} with finite mass and finite energy, such that $\rho>0$. Then for any $0<T<\infty$  we have
\begin{equation*}
\sup_{t\in[0, T]}I(t)\leq C(T)I(0).
\end{equation*}
\end{prop}

\begin{proof}
Because of positivity and smoothness of solutions to \eqref{eq:QHD_1d}, we can rewrite the functional $I(t)$ in the following way,
\begin{equation*}
I(t)=\frac12\int_\R\rho(\mu^2+\sigma^2)\,dx,
\end{equation*}
where $\mu$ is the chemical potential and $\sigma=\d_t\log\sqrt{\rho}$. 
Now, by using the identity
\begin{equation*}
\rho\mu=-\frac14\d_x^2\rho+e+p(\rho)
\end{equation*}
we have
\begin{equation*}
\begin{aligned}
\d_t(\rho\mu)=&\d_t\left(e-\frac14\d_x^2\rho+p(\rho)\right)\\
=&\d_x(\d_x\sqrt{\rho}\d_t\sqrt{\rho}-\rho v\mu)-\frac14\d_{t}\d_x^2\rho+\d_tp(\rho).
\end{aligned}
\end{equation*}
Again by using the continuity equation we have
\begin{equation}\label{eq:mu_evol}
\rho\d_t\mu+\rho v\d_x\mu=\d_x(\d_x\sqrt{\rho}\d_t\sqrt{\rho})-\frac14\d_{t}\d_x^2\rho+\d_tp(\rho).
\end{equation}
Now we are going to derive the equation for $\sigma$. By using the continuity equation, the equation for $\log\sqrt{\rho}$ is given by
\begin{equation*}
\d_t\log\sqrt{\rho}+v\d_x\log\sqrt{\rho}+\frac12\d_xv=0,
\end{equation*}
then by using the identity $\d_tv=-\d_x\mu$, we obtain
\begin{equation*}
\d_t\sigma+v\d_x\sigma-\d_x\mu\d_x\log\sqrt{\rho}-\frac12\d_x^2\mu=0,
\end{equation*}
hence multiplying it by $\rho$ we get
\begin{equation}\label{eq:sigma_evol}
\rho\d_t\sigma+\rho v\d_x\sigma=\frac12\d_x(\rho\d_x\mu).
\end{equation}
Then by using \eqref{eq:mu_evol} and \eqref{eq:sigma_evol}, it follows that
\begin{align*}
\frac{d}{dt}I(t)=\int\mu\d_tp(\rho)\,dx=&2\int p'(\rho)\d_t\sqrt{\rho}\lambda\,dx\\
\lesssim&\|\sqrt{\rho}(t)\|_{L^\infty}^{\gamma-1}I(t)\\
\lesssim & \|\sqrt{\rho}(t)\|_{H^1_x}^{\gamma-1}I(t)\leq [M(t)+E(t)]^{\frac{\gamma-1}{2}}I(t).
\end{align*}
Then Gronwall's inequality gives the desired result.
\end{proof}

\begin{rem}\label{rmk:higher_schr}
As for Lemma \ref{lem:en}, the same result can be obtained also for the solutions constructed in Theorem \ref{thm:glob2}. Indeed, since they are Schr\"odinger-generated, we have
\begin{equation*}
\frac{d}{dt}I(t)=-\int\d_tf'(\rho)\IM(\bar\psi\d_t\psi)\,dx=\int\d_tf'(\rho)\sqrt{\rho}\lambda\,dx.
\end{equation*}
This fact, together with the stability result proved in the next section, suggests that there are different ways to approximate the system \eqref{eq:QHD_1d} and to obtain different classes of weak solutions.
\\
In addition, let us remark that similar functional are also used in the context of the Schrödinger equation to study the growth of Sobolev norms \cite{PTV}.
\end{rem}

In the remaining part of this section we show that by using an approach similar to the one we used before to prove the Morawetz estimate in Theorem \ref{thm:mor} and by assuming that $(\rho,J)$ satsfies the weak entropy inequality as in the Definition \ref{def:entrsln}, it is possible to infer an improved space-time bound for the energy density.

\begin{prop}\label{prop:unif0}
Let $(\rho,J)$ be a GCP weak solution to \eqref{eq:QHD_1d} in the sense of Definition \ref{def:cptsln}, namely 
\begin{equation}\label{eq:unif0_prop}
\begin{aligned}
&\|\sqrt{\rho}\|_{L^\infty(0, T;H^1(\R))}+\|\Lambda\|_{L^\infty(0, T;L^2(\R))}\leq M_1\\
&\|\d_t\sqrt{\rho}\|_{L^{\infty}(0, T;L^2(\R))}+\|\lambda\|_{L^\infty(0, T;L^2(\R))}\leq M_2.
\end{aligned}
\end{equation}
Furthermore let us assume that the weak entropy inequality 
\begin{equation}\label{eq:enineq}
\d_te+\d_x(\Lambda\lambda-\d_t\sqrt{\rho}\d_x\sqrt{\rho})\le0
\end{equation}
is satisfied in the sense of  distribution. Then we have 
\begin{equation}\label{eq:kinL2tx}
\|e\|_{L^2_{t,x}}+\|\d_x^2\rho\|_{L^2_{t,x}}+\|\d_x J\|_{L^\infty_t L^2_x}\leq C(M_1,M_2)(1+T)^\frac12.
\end{equation}
\end{prop}

\begin{proof}
By Sobolev embedding, \eqref{eq:unif0_prop} implies $\rho$ is continuous in space and $\|\sqrt{\rho}\|_{L^\infty_{t,x}}\le C(M_1)$. Then by using the equation \eqref{eq:QHD_1d} we have
\[
\|\d_x J\|_{L^\infty_t L^2_x}=\|\d_t \rho\|_{L^\infty_t L^2_x} \le 2\|\sqrt{\rho}\|_{L^\infty_{t,x}} \|\d_t \sqrt{\rho}\|_{L^\infty_t L^2_x} \le C(M_1,M_2).
\]
As in the proof of Proposition \ref{thm:mor} we have $\|J\|_{L^\infty_tL^1_x}\leq M_1^2$. Thus, by choosing $C_1>0$ such that $C_1>\|J\|_{L^\infty_tL^1_x}$, we have that the anti-derivative of J defined by
\begin{equation*}
G(t,x)=\int_{-\infty}^x J(t,s)ds+C_1,
\end{equation*}
satisfies $G\geq0$ and moreover $\|G\|_{L^\infty_{t, x}}\leq C(\|J\|_{L^\infty_tL^1_x})$.
By integrating in space the momentum equation in \eqref{eq:QHD_1d}, it follows that
\begin{equation*}
\d_tG+\Lambda^2+(\d_x\sqrt\rho)^2+p(\rho)=\frac14\d_x^2\rho,
\end{equation*}
namely
\begin{equation}\label{eq:eqG2}
\d_tG+(\sqrt{\rho}\lambda+e)-2f(\rho)=0.
\end{equation}
The bounds in \eqref{eq:unif0_prop} imply the equation \eqref{eq:eqG2} holds in the space $L^\infty_t L^1_x$. 
\newline
Analogously to Theorem \ref{thm:mor}, here we want to consider the following interaction functional
\begin{equation*}
\int_{\R}G(t, x)e(t, x)\,dx
\end{equation*}
and compute its time derivative. However, since the entropy inequality \eqref{eq:enineq} only holds in the sense of distribution, we introduce the following mollification   to rigorously justify the argument. Let $\{\chi_\eps\}_{\eps>0}$ be a sequence of smooth, positive space-time mollifiers supported on $(-\eps,\eps)^2$, and for any function $g\in L^1([0,T);L^1_{loc}(\R))$ we define 
$$g_\eps(t,x)=\int_0^T\int_\R g(t-s,x-y)\chi_\eps(s,y)dyds,\quad (t,x)\in (\eps,T-\eps)\times\R.$$
By taking the convolution of \eqref{eq:enineq} and \eqref{eq:eqG2} with $\chi_\eps$, we obtain
\begin{equation}\label{eq:enineqeps}
\d_te_\eps+\d_x(\Lambda\lambda-\d_t\sqrt{\rho}\d_x\sqrt{\rho})_\eps\le0
\end{equation}
and
\begin{equation}\label{eq:eqG2eps}
\d_tG_\eps+(\sqrt{\rho}\lambda+e)_\eps-2f(\rho)_\eps=0.
\end{equation}

Let us multiply \eqref{eq:eqG2eps} by $(\sqrt{\rho}\lambda+e)_\eps$, then
\begin{equation}\label{eq:eq7}
\begin{aligned}
\int_{\eps}^{T-\eps}\int_\R (\sqrt\rho\lambda+e)^2_\eps dxdt=&-\int_{\eps}^{T-\eps}\int_\R (\sqrt\rho\lambda+e)_\eps\d_tG_\eps\,dxdt\\
&+2\int_{\eps}^{T-\eps}\int_\R (\sqrt\rho\lambda+e)_\eps f(\rho)_\eps dxdt.
\end{aligned}
\end{equation}
We estimate the right hand side of \eqref{eq:eq7} in the following way. 

The first term on the right hand side satisfies
\begin{align*}
-\int_{\eps}^{T-\eps}\int_\R (\sqrt\rho\lambda+e)_\eps\d_tG_\eps dxdt=&-\left.\int_\R (\sqrt\rho\lambda+e)_\eps(t)G_\eps(t)dx\right|_{t=\eps}^{t=T-\eps}\\
&+\int_{\eps}^{T-\eps}\int_\R G_\eps\d_t(\sqrt\rho\lambda+e)_\eps dxdt\\
=&-\left.\int_\R (\sqrt\rho\lambda+e)_\eps(t)G_\eps(t)dx\right|_{t=\eps}^{t=T-\eps}\\
&+\int_{\eps}^{T-\eps}\int_\R G_\eps\d_t(-\frac14\d_x^2\rho_\eps+2e_\eps+p(\rho)_\eps)dxdt,
\end{align*}
where the boundary terms can be estimated as follows
\begin{align*}
\int_\R (\sqrt\rho\lambda+e)_\eps(t)G_\eps(t)dx\le&\|G\|_{L^\infty_{t,x}}(\|e\|_{L^\infty_{t}L^1_x}+\|\sqrt\rho\|_{L^\infty_{t}L^2_x}\|\lambda\|_{L^\infty_{t}L^2_x})\le C(M_1,M_2).
\end{align*}
Now we estimate the integral 
\begin{equation}\label{eq:eq8}
\int_{\eps}^{T-\eps}\int_\R G_\eps\d_t(-\frac14\d_x^2\rho_\eps+2e_\eps+p(\rho)_\eps)dxdt.
\end{equation}
By using the equation of $\rho$ in \eqref{eq:QHD_1d} and integration by parts in space, we have 
\begin{align*}
\int_{\eps}^{T-\eps}\int_\R G_\eps\d_t(-\frac14\d_x^2\rho_\eps)dxdt=&-\frac14\int_{\eps}^{T-\eps}\int_\R \d_x^2G_\eps\d_t\rho_\eps\,dxdt\\
=&\frac14\int_{\eps}^{T-\eps}\int_\R (\d_xJ_\eps)^2dxdt\leq C(M_1, M_2) T. 
\end{align*} 
By using the mollified entropy inequality \eqref{eq:enineqeps} and $G_\eps\ge 0$, we obtain
\begin{align*}
\int_{\eps}^{T-\eps}\int_\R G_\eps\,\d_t e_\eps\,dxdt\le&-\int_{\eps}^{T-\eps}\int_\R G_\eps\,\d_x(\lambda\Lambda-\d_x\sqrt\rho\d_t\sqrt\rho)_\eps\,dxdt\\\
=&\int_{\eps}^{T-\eps}\int_\R J_\eps\,(\lambda\Lambda-\d_x\sqrt\rho\d_t\sqrt\rho)_\eps\,dxdt\\
\leq & T\,\|J\|_{L^\infty_{t,x}}\left(\|\lambda\|_{L^\infty_t L^2_x}\|\Lambda\|_{L^\infty_t L^2_x}\right.\\
&\quad +\left.\|\d_x\sqrt\rho\|_{L^\infty_t L^2_x}\|\d_t\sqrt\rho\|_{L^\infty_t L^2_x}\right)\leq C(M_1,M_2)T.
\end{align*}
Next we have 
\begin{align*}
\int_{\eps}^{T-\eps}\int_\R G_\eps\d_t p(\rho)_\eps dxdt=&2\int_{\eps}^{T-\eps}\int_\R G_\eps[p'(\rho)\sqrt\rho\d_t\sqrt\rho]_\eps dxdt\\
\le & 2T\|G\|_{L^\infty_{t,x}}\|p'(\rho)\sqrt\rho\|_{L^\infty_t L^2_x}\|\d_t\sqrt\rho\|_{L^\infty_t L^2_x}\le C(M_1,M_2)T.
\end{align*}
 Thus we control \eqref{eq:eq8} by $C(M_1,M_2)T$.

For the last term on the right hand side of \eqref{eq:eq7}, by using the $L^\infty_{t,x}$ bound of $\sqrt{\rho}$ and \eqref{eq:unif0_prop}, it is straightforward to obtain that
\[
\int_{\eps}^{T-\eps}\int_\R (\sqrt\rho\lambda+e)_\eps f(\rho)_\eps dxdt\le C(M_1,M_2)T.
\]

Summarising the previous estimates we obtain
\begin{equation*}
\int_{\eps}^{T-\eps}\int_\R (\sqrt\rho\lambda+e)^2_\eps\, dxdt\leq C(M_1,M_2)(1+T),
\end{equation*}
uniformly in $\eps>0$.
By combining this bound together with  
$(\sqrt{\rho}\lambda+e)_\eps\to\sqrt{\rho}\lambda+e$ in $L^\infty_tL^1_x$, as $\eps\to0$, we finally obtain
\begin{equation*}
\|\sqrt{\rho}\lambda+e\|_{L^2_{t,x}}^2\leq \liminf\limits_{\eps\to 0}\int_{\eps}^{T-\eps}\int_\R (\sqrt\rho\lambda+e)^2_\eps \,dxdt\leq C(M_1,M_2)(1+T).
\end{equation*}
Therefore 
\begin{align*}
\|e\|_{L^2_{t,x}}^2\le & \|\sqrt{\rho}\lambda+e\|_{L^2_{t,x}}^2+\|\sqrt{\rho}\lambda\|_{L^2_{t,x}}^2\\
\le & \|\sqrt{\rho}\lambda+e\|_{L^2_{t,x}}^2+T\|\sqrt{\rho}\|_{L^\infty_{t,x}}^2\|\lambda\|_{L^\infty_t L^2_x}^2\le C(M_1,M_2)(1+T).
\end{align*}

We conclude our argument by estimating $\d_x^2\rho$. By using the definition of $\lambda$ and the previous bounds, it follows that
\begin{align*}
\|\d_x^2\rho\|_{L^2_{t,x}}^2\le & 4\|\sqrt\rho\lambda+e\|_{L^2_{t,x}}^2+4\|p(\rho)\|_{L^2_{t,x}}^2
\le C(M_1,M_2)(1+T).
\end{align*}
\end{proof}

\section{Stability}\label{sect:comp}

In this section we are going to prove Theorem \ref{thm:stab}, namely that for any uniformly bounded sequence of GCP weak solutions in the sense of Definition \ref{def:cptsln}, it is possible to infer suitable stability properties. Let us remark here that in general we don't know whether the solutions constructed in Theorem \ref{thm:glob} and \ref{thm:glob2} are the only possible ones, as there could be alternative methods to construct weak solutions to \eqref{eq:QHD_1d}. However for the classes of weak solutions we consider in this paper, the energy \eqref{eq:en_QHD} and the functional \eqref{eq:higher} are always uniformly bounded along the flow of solutions to \eqref{eq:QHD_1d}, and it is also true for smooth solutions as shown in Proposition \ref{prop:higher}.

To prove Theorem \ref{thm:stab} we consider a sequence of GCP weak solutions satisfying the following uniform bounds
\begin{equation}\label{eq:261}
\begin{aligned}
&\|\sqrt{\rho_n}\|_{L^\infty(0, T;H^1(\R))}+\|\Lambda_n\|_{L^\infty(0, T;L^2(\R))}\leq M_1\\
&\|\d_t\sqrt{\rho_n}\|_{L^\infty(0, T;L^2(\R))}+\|\lambda_n\|_{L^\infty(0, T;L^2(\R))}\leq M_2,
\end{aligned}
\end{equation}
where $\lambda_n$ is such that 
\begin{equation}\label{eq:chem_n}
\sqrt{\rho_n}\lambda_n=-\frac14\d_x^2\rho_n + e_n + p(\rho_n)
\end{equation}
and 
\begin{equation}\label{eq:en_dens_n}
e_n=\frac12(\d_x\sqrt{\rho_n})^2+\frac12\Lambda_n^2+f(\rho_n).
\end{equation}
Besides the above uniform bounds we need some further assumptions in order to infer the compactness for $\{\d_x\sqrt{\rho_n}\}$ and $\{\Lambda_n\}$. More precisely we are going to assume that one of the three assumptions stated in Theorem \ref{thm:stab} holds true. Here we start by considering the assumption (3), namely we assume that the sequence  $\{(\sqrt{\rho_n}, \Lambda_n)\}$ satisfies the weak entropy inequality
\begin{equation}\label{eq:enineq_n}
\d_te_n+\d_x(\Lambda_n\lambda_n-\d_x\sqrt\rho_n\d_t\sqrt\rho_n)\leq 0
\end{equation}
for all $n\ge 1$. By using Proposition \ref{prop:unif0} we have the additional estimates
\begin{equation}\label{eq:unif3}
\|e_n\|_{L^2_{t,x}}+\|\d_x^2\rho_n\|_{L^2_{t,x}}+\|\d_x J_n\|_{L^\infty_t L^2_x}\leq C(1+T)^\frac12.
\end{equation}
By collecting the bounds \eqref{eq:261} and \eqref{eq:unif3} we have, up to subsequences,
\begin{align}\label{eq:weak1}
\sqrt{\rho_n}\rightharpoonup&\sqrt{\rho},\qquad L^\infty(0, T;H^1(\R))\cap W^{1, \infty}(0, T;L^2(\R)),\\
\label{eq:weak2}
\Lambda_n\rightharpoonup&\Lambda,\qquad L^\infty(0, T;L^2(\R)),
\end{align}
and for the energy density
\begin{equation}\label{eq:weak3}
e_n\rightharpoonup\nu,\qquad L^2(0, T;L^2(\R)).
\end{equation} 
One of the main difficulties to show our stability result is to verify that \eqref{eq:en_dens_n} holds also in the limit, namely that we have
\begin{equation*}
\nu=\frac12(\d_x\sqrt{\rho})^2+\frac12\Lambda^2+f(\rho).
\end{equation*}

The next Proposition shows that \eqref{eq:en_dens_n} indeed holds in the limiting case, and $(\sqrt{\rho_n},\Lambda_n)$ converges strongly. Furthermore, we prove that $\nu$ and $\Lambda$ vanish almost everywhere in the vacuum set $\{\rho=0\}$, which matches with the physical interpretation of the energy density.

\begin{prop}\label{prop:strong1}
Let $\{(\rho_n,J_n)\}$ be a sequence of GCP weak solutions to \eqref{eq:QHD_1d} satisfying the weak entropy inequality \eqref{eq:enineq_n}. Then for the weak limits \eqref{eq:weak1}, \eqref{eq:weak2} and \eqref{eq:weak3} we have the following nonlinear identity
\begin{equation*}
\nu=\frac12(\d_x\sqrt{\rho})^2+\frac12\Lambda^2+f(\rho)
\end{equation*}
for a.e. $(t,x)\in[0,T]\times\R$, and $\nu=0$ a.e. $(t,x)\in\{\rho=0\}$. Furthermore we have the following local strong convergence results
\begin{equation*}
\begin{aligned}
\d_x\sqrt{\rho_n}&\to \d_x\sqrt{\rho},&L^2(0, T;L^2_{loc}(\R)),\\
\Lambda_n&\to\Lambda,&L^2(0, T;L^2_{loc}(\R)).
\end{aligned}
\end{equation*}
\end{prop}

\begin{proof}
By using the identity \eqref{eq:en_dens_n}, we have 
\begin{equation}\label{eq:101}
\rho_n e_n=\frac18(\d_x\rho_n)^2+\frac12J_n^2+f(\rho_n)\rho_n.
\end{equation}
From the uniform bounds \eqref{eq:unif3} it follows that, up to subsequences,
\begin{equation}\label{eq:strong1}
\begin{aligned}
\rho_n\to&\rho, \quad L^2(0,T;H^1_{loc}(\R))\\
J_n\to&J,\quad L^2(0,T;L^2_{loc}(\R)),
\end{aligned}
\end{equation}
hence, by passing to the limit in \eqref{eq:101} we obtain
\begin{equation}\label{eq:103}
\rho\nu=\frac18(\d_x\rho)^2+\frac12J^2+f(\rho),\quad\textrm{a.e.}
\end{equation}
Since $J_n=\sqrt{\rho_n}\Lambda_n$, 
then again by \eqref{eq:weak2} and \eqref{eq:strong1} we have $J=\sqrt{\rho}\Lambda$ a.e. in $\R$,
therefore that we can write \eqref{eq:103} as 
\begin{equation}\label{eq:102}
\rho\left(\nu-\frac12(\d_x\sqrt{\rho})^2-\frac12\Lambda^2-f(\rho)\right)=0,\quad\textrm{a.e.}
\end{equation}
We now claim that $\nu=0$ a.e. on $\{\rho=0\}$. By passing to the limit in the identity \eqref{eq:chem_n} we obtain
\begin{equation*}
\nu=\sqrt{\rho}\lambda+\frac14\d_x^2\rho-p(\rho).
\end{equation*}
Again by Lemma \ref{lemma:LL} we have $\d_x^2\rho=0$ a.e. on $\{\rho=0\}$, consequently $\nu=0$ a.e. on $\{\rho=0\}$.
For any $R>0$ let us define $V_R=\{\rho=0\}\cap([0,T]\times(-R, R))$, then by \eqref{eq:261}, \eqref{eq:en_dens_n} and \eqref{eq:weak3} it follows that
\begin{align*}
\int\int_{V_R}\frac12(\d_x\sqrt{\rho})^2+\frac12\Lambda^2\,dxdt\leq&\liminf_{n\to\infty}\int\int_{V_R}\frac12(\d_x\sqrt{\rho_n})^2+\frac12\Lambda_n^2\,dxdt\\
=&\lim_{n\to\infty}\int\int_{V_R}e_n-f(\rho_n)\,dxdt\\
=&\int\int_{V_R}\nu-f(\rho)\,dxdt=0.
\end{align*}
The previous inequality implies that $\Lambda=0$ a.e. on $\{\rho=0\}$ and consequently from \eqref{eq:102} we also have 
\begin{equation*}
\nu=\frac12(\d_x\sqrt{\rho})^2+\frac12\Lambda^2+f(\rho).
\end{equation*}
We conclude by proving the local strong convergence. Because of the previous results, for any $R>0$ it follows that
\begin{equation*}
\lim_{n\to\infty}\int_0^T\int_{-R}^R e_n-f(\rho_n)dxdt=\int_0^T\int_{-R}^R\nu-f(\rho)\,dxdt,
\end{equation*}
therefore 
\begin{align*}
\int_0^T\int_{-R}^R\frac12(\d_x\sqrt{\rho})^2+\frac12\Lambda^2\,dxdt&\leq\liminf_{n\to\infty}\int_0^T\int_{-R}^R\frac12(\d_x\sqrt{\rho_n})^2+\frac12\Lambda_n^2\,dxdt\\
&=\lim_{n\to\infty}\int_0^T\int_{-R}^R e_n-f(\rho_n)\,dxdt\\
&=\int_0^T\int_{-R}^R\nu-f(\rho)\,dxdt\\
&=\int_0^T\int_{-R}^R\frac12(\d_x\sqrt{\rho})^2+\frac12\Lambda^2\,dxdt,
\end{align*}
and we conclude 
\begin{equation*}
\|\d_x\sqrt{\rho}\|_{L_{t}^{2}L_{x, loc}^{2}}=\lim_{n\to\infty}\|\d_x\sqrt{\rho_n}\|_{L_{t}^{2}L_{x, loc}^{2}},\quad \|\Lambda\|_{L_{t}^{2}L_{x, loc}^{2}}=\lim_{n\to\infty}\|\Lambda_n\|_{L_{t}^{2}L_{x, loc}^{2}}.
\end{equation*}
Hence by using the convergence of norms we can improve the weak convergence to a strong one.
\end{proof}

In what follows we are going to show that we are able to prove the analogue of Proposition \ref{prop:strong1} even under the assumptions (1) or (2) of Theorem \ref{thm:stab}, which is a direct consequence of the wave function lifting method established in the Section \ref{sect:lift}.

\begin{prop}\label{unif2}
Let $\{(\rho_n,J_n)\}$ be a sequence of GCP weak solutions to \eqref{eq:QHD_1d}. 
Let us further assume that one of the following conditions holds true for the sequence $\{(\rho_n, J_n)\}$:
\begin{itemize}
\item[(1)] for almost every $t\in [0,T]$, $\rho_n(t,\cdot)>0$;
\item[(2)] for almost every $t\in [0,T]$, $e_n(t,\cdot)$ is continuous.
\end{itemize} 
Then we have 
\begin{equation}\label{eq:unif2}
\|\d_x^2\rho_n\|_{L^\infty_t L^2_x}+\|\d_x J_n\|_{L^\infty_t L^2_x}+\|\d_x \sqrt{e_n}\|_{L^\infty_t L^2_x}\leq C(M_1,M_2).
\end{equation}
\end{prop}

\begin{proof}
Let $(\rho_n,J_n)$ be a GCP solution, then the bounds $\eqref{eq:261}$ are satisfied, and let us assume that $e_n(t,\cdot)$ is continuous for a.e. $t\in[0,T)$. Then it is straightforward to see that for a.e. $t\in[0,T)$, the hydrodynamic state $(\sqrt{\rho_n}(t,\cdot),\Lambda_n(t,\cdot))$ is a GCP state. Therefore there exists a wave function $\psi_n(t, \cdot)\in H^2(\R)$ associated to $(\sqrt{\rho_n}(t,\cdot),\Lambda_n(t,\cdot))$ in the sense of the Definition \ref{def:ass}, moreover we have
\[
\|\psi_n(t, \cdot)\|_{H^2(\R)}\leq C(M_1,M_2).
\]
The estimates \eqref{eq:unif2} is a direct consequence of the $H^2$ regularity of $\psi_n$ and the polar factorization.

If we assume the condition that $\rho_n(t,\cdot)>0$ for a.e. $t\in[0,T)$, by Lemma \ref{lemma:H2} it follows that the wave function lifting argument still holds true.
\end{proof}

Let us remark that in the proof of Proposition \ref{unif2}, the associated wave function $\psi_n(x;t)$ is constructed for isolated $t$, and in general it is not a solution to the Schr\"odinger equation.
On the other hand, for solutions like the ones constructed in Theorem \ref{thm:glob2}, the bounds in \eqref{eq:unif2} come straightforwardly, see \eqref{eq:B2_c}. The bound of $\sqrt{e_n}$ in \eqref{eq:unif2} implies local strong convergence 
\begin{equation}\label{eq:locstrong1}
\sqrt{e_n}\rightarrow\omega,\qquad L^p(0, T;L^2_{loc}(\R)),
\end{equation}
for $2\leq p<\infty$ and we also have the compactness proposition. Proposition \ref{prop:strong2} can be proved by same argument as Proposition \ref{prop:strong1} with minor modification.

\begin{prop}\label{prop:strong2}
Let $\{(\sqrt\rho_n,\Lambda_n)\}$ be a sequence of weak solution to \eqref{eq:QHD_1d} in the GCP class. 
Let us further assume that one of the following conditions holds true for the sequence $\{(\sqrt{\rho_n}, \Lambda_n)\}$:
\begin{itemize}
\item[(1)] for almost every $t\in [0,T]$, $\rho_n(t,\cdot)>0$;
\item[(2)] for almost every $t\in [0,T]$, $e_n(t,\cdot)$ is continuous.
\end{itemize}
Then the weak limits \eqref{eq:weak1}, \eqref{eq:weak2} and the local strong limit \eqref{eq:locstrong1} satisfy
\begin{equation*}
\omega^2=\frac12(\d_x\sqrt{\rho})^2+\frac12\Lambda^2+f(\rho)
\end{equation*}
for a.e. $(t,x)\in[0\,T]\times\R$, and $\omega=0$ a.e. $(t,x)\in\{\rho=0\}$. Furthermore we have the following local strong convergence
\begin{equation*}
\begin{aligned}
\d_x\sqrt{\rho_n}&\to \d_x\sqrt{\rho},&L^p(0, T;L^2_{loc}(\R)),\\
\Lambda_n&\to\Lambda,&L^p(0, T;L^2_{loc}(\R)),
\end{aligned}
\end{equation*}
for $2\le p <\infty$.
\end{prop}

Finally, we conclude this section by proving Theorem \ref{thm:stab}.

\begin{proof}[Proof of Theorem \ref{thm:stab}]
To finish the proof of Theorem \ref{thm:stab}, we only need to
show that the strong limit $(\sqrt{\rho}, \Lambda)$ obtained in Proposition \ref{prop:strong1} and Proposition \ref{prop:strong2} is indeed a weak solution of the QHD system. This is indeed a direct consequence of the compactness of $\{(\sqrt\rho_n,\Lambda_n)\}$.
\end{proof}
\section{An hydrodynamic definition of \texorpdfstring{$\lambda$}{lambda} and its relations with the chemical potential}\label{sect:lambda}
In what follows we are going to provide some heuristic motivations in order to state a rigorous definition, independent of the Madelung approach and given in the terms of the hydrodynamic variables, of the function $\lambda$ introduced in \eqref{eq:higher}. In the case of hydrodynamic solutions to the QHD system generated by a $H^2$ wave function through the Madelung transform, the functional $I(t)$  is given by 
\begin{equation*}
I(t)=\int|\d_t\psi|^2\,dx.
\end{equation*}
As a consequence of the polar factorization Lemma \ref{lemma:polar} we have $\d_t\sqrt{\rho}=\RE(\bar\phi\d_t\psi)$. If we write $\lambda=-\IM(\bar\phi\d_t\psi)$, then it follows that
\begin{equation}\label{eq:polar_H2}
|\d_t\psi|^2=(\d_t\sqrt{\rho})^2+\lambda^2,\quad\textrm{a.e. in}\;\R.
\end{equation}
Clearly since $\psi\in L^\infty_tH^2_x$, we have $\lambda\in L^\infty_tL^2_x$.
Moreover by using the regularity of $\psi$ we also have that $\lambda=0$ a.e. in the vacuum set $\{\rho=0\}$. 

However it is not possible to use the identity $\lambda=-\IM(\bar\phi\d_t\psi)$ to single out a definition given only in hydrodynamic terms. Indeed in order to exploit the previous identity it would be necessary to use distributional derivatives of the polar factor.
On the other hand, from the Schr\"odinger equation \eqref{eq:NLS_1d}, it is straightforward to define an auxiliary quantity $\xi$ which represents the chemical potential $\mu$, defined in \eqref{eq:chem}, in the quantum probability measure $\rho dx$,
\begin{equation*}
\xi =\rho \mu= \sqrt{\rho}\lambda = - \IM(\bar\psi\d_t\psi) = -\frac14\d_x^2\rho+\frac12(\d_x\sqrt\rho)^2+\frac12\Lambda^2+f'(\rho)\rho.
\end{equation*} 
By assuming $\d_x^2\rho\in L^1_{loc}$ and $\Lambda = 0$ a.e. on $\{\rho = 0\}$, we have that
\begin{equation}\label{eq:xi}
\xi = -\frac14\d_x^2\rho+\frac12(\d_x\sqrt\rho)^2+\frac12\Lambda^2+f'(\rho)\rho
\end{equation}
is a well-defined function $\xi$ in $L^1_{loc}$ and moreover $\xi = 0$ a.e. on $\{\rho = 0\}$. 

Therefore $\xi/\sqrt{\rho}$ belongs to $L^1_{loc}(\sqrt{\rho}dx)$,
so it is possible to define a.e. a function in $L^1_{loc}({\{\rho>0\}})$ and $L^1_{loc}$ in the interior of the vacuum set (but not $L^1_{loc}(\R)$)
\begin{equation*}
\lambda=\left\{\begin{array}{cc}
\frac{\xi}{\sqrt\rho}&\textrm{in }\;\{\rho>0\}\\
0&\textrm{elsewhere},
\end{array}\right.
\end{equation*}
where we set $\lambda = 0$ in the vacuum region to be consistent with the case when the hydrodynamic quantities are given by Madelung transforms.  By using the identity \eqref{eq:xi}, outside the vacuum set, namely on the open set 
$\{\rho>0\},$ we have a.e.
\begin{equation*}\label{eq:lambda}
\lambda=-\frac12\d_{x}^2\sqrt{\rho}
+\frac12\frac{\Lambda^2}{\sqrt{\rho}}
+\sqrt{\rho}f'(\rho),
\end{equation*}
where $\d_x^2\sqrt{\rho}$ is intended as
\[
\d_x^2\sqrt{\rho}=\frac{\d_x^2\rho}{2\sqrt{\rho}}-\frac{(\d_x\sqrt\rho)^2}{\sqrt\rho}.
\] 
This identity allows us to provide the Definition \ref{def:lambda} where $\lambda$ is defined rigorously only in terms of the hydrodynamic variables.

We wish to emphasize that $\d_x^2\sqrt{\rho}$ and $\Lambda^2/\sqrt{\rho}$ may become singular when approaching vacuum boundaries. In what follows, we provide a series of Propositions characterising the properties of $\d_{x}^2\sqrt{\rho}$, $\Lambda^2/\sqrt{\rho}$ for GCP hydrodynamic states, see the Definition \ref{def:GCP}.
More precisely, in what follow we will consider a hydrodynamic state $(\sqrt{\rho}, \Lambda)$ satisfying
\begin{itemize}
\item $\Lambda=0$ a.e. on $\{\rho=0\}$;
\item $\d_x^2\rho, \d_xJ\in L^1_{loc}(\R)$;
\item the following bounds are satisfied
\begin{equation}\label{eq:GCP_bounds}
\begin{aligned}
\|\sqrt{\rho}\|_{H^1(\R)}+\|\Lambda\|_{L^2(\R)}\leq &M_1\\
\|\mathbf{1}_{\{\rho>0\}}\d_x J/\sqrt{\rho}\|_{L^2(\R)}+\|\lambda\|_{L^2(\R)}\leq &M_2,
\end{aligned}
\end{equation}
where $\lambda$ is defined as in \eqref{eq:def_lambda}.
\end{itemize}
The main result here, in the sense of studying the singularity at vacuum boundaries, is stated in Proposition \ref{prop:integrability} and Proposition \ref{prop:lambda} below. More presicely, for the functions $\d_x^2\sqrt{\rho}$ and $\Lambda^2/\sqrt{\rho}$, which are only assumed to be integrable with respect to the density measure, we strengthen their integrability to one in the Lebesgue sense upto vacuum boundaries. As a consequence, it shows that the only possible singularity of a GCP state $(\sqrt\rho,\Lambda)$ is the  jumps of $\d_x\sqrt\rho$ and $\Lambda$ at vacuum boundaries, and $\lambda$ can be extended to a Radon measure 
\[
\tilde\lambda=-\frac12\d_{x}^2\sqrt{\rho}
+\frac12\frac{\Lambda^2}{\sqrt{\rho}}\mathbf{1}_{\{\rho>0\}}
+\sqrt{\rho}f'(\rho) \in \mathcal{D'}(\R),
\]
which takes into account the possible singularities at vacuum boundaries.

By the continuity of $\rho$ the non vacuum set is at most a countable union of disjoint intervals, namely
\begin{equation}\label{eq:nonvac}
V^c=\{x\in\R;\rho(x)> 0\}=\underset{j}{\cup}(a_j,b_j).
\end{equation}
Notice that there are at most two intervals with infinite length. These intervals will be denoted as $(-\infty,b_{-\infty})$ and $(a_\infty,\infty)$ and they will be treated separately if necessary. On any interval outside the vacuum, $(a,b)\subset V^c$, we have the following proposition:

\begin{prop}\label{prop:inftybd}
Let $(\sqrt{\rho}, \Lambda)$ be a GCP state satisfying the bounds \eqref{eq:GCP_bounds}, then we have
\begin{equation}\label{eq:eq4}
\d_x e=2\d_x\sqrt{\rho}(2f'(\sqrt{\rho})\sqrt{\rho})-\lambda)-\Lambda\frac{\d_xJ}{\sqrt{\rho}}\quad in\; \mathcal D'(a, b),
\end{equation}
where $e$ denotes the total energy density given by \eqref{eq:en_dens}.
Moreover, 
\begin{equation}\label{eq:nonvacbd}
\begin{aligned}
&\|\d_x\sqrt{e}\|_{L^2(a,b)}\le  C(M_1,M_2),\\
&\|\sqrt{e}\|_{L^\infty(a,b)}\le  C(M_1,M_2)(1+|b-a|^{-\frac12}).
\end{aligned}
\end{equation}
\end{prop}

\begin{proof}
Let us consider the interval $(a,b)\subset V^c$, then
\[
\d_x^2\sqrt{\rho}=(\frac{1}{2}\d_x^2\rho-(\d_x\sqrt{\rho})^2)/\sqrt{\rho}\quad in\; L^1_{loc}(a, b),
\]
and 
\[
\d_x\Lambda=(\d_x J-\Lambda\d_x\sqrt{\rho})/\sqrt{\rho}\quad in\; L^1_{loc}(a, b).
\]
As a consequence $\d_x\sqrt{\rho}, \Lambda\in \mathcal C(a, b)$. Then by using the definitions \eqref{eq:en_dens} and \eqref{eq:def_lambda} we have the following identities
\begin{align*}
\d_x e
=&\d_x\sqrt{\rho}\d^2_x\sqrt{\rho}+\Lambda\d_x\Lambda+f'(\rho)\d_x\rho\\
=&\d_x\sqrt{\rho}(\d_x^2\sqrt{\rho}-\frac{\Lambda^2}{\sqrt{\rho}}+2f'(\rho)\sqrt{\rho})+\d_x\sqrt{\rho}\frac{\Lambda^2}{\sqrt{\rho}}+\Lambda\d_x\Lambda\\
=&2\d_x\sqrt{\rho}(2f'(\rho)\sqrt{\rho}-\lambda)+\Lambda\frac{\d_x J}{\sqrt{\rho}}.
\end{align*}

To show the first inequality of \eqref{eq:nonvacbd}, we divide both side of \eqref{eq:eq4} by $2\sqrt{e+\epsilon}$, with $\epsilon >0$. We then obtain
\begin{align*}
\d_x \sqrt{e+\epsilon}
&=\frac{\d_x e}{2\sqrt{e+\epsilon}}\\
&=\frac{\d_x\sqrt{\rho}}{\sqrt{e+\epsilon}}(2f'(\rho)\sqrt{\rho}-\lambda)
-\frac{\Lambda}{\sqrt{e+\epsilon}}\frac{\d_xJ}{2\sqrt{\rho}}.
\end{align*}
By taking the $L^2_x$ norm over $(a,b)$, and using the bounds \eqref{eq:GCP_bounds} we get 
\begin{align*}
\|\d_x\sqrt{e+\epsilon}\|_{L^2(a,b)}\le& \|\frac{\d_x\sqrt{\rho}}{\sqrt{e+\epsilon}}\|_{L^\infty(a,b)}\|2f'(\rho)\sqrt{\rho}-\lambda\|_{L^2(a_j,b_j)}\\
&+ \|\frac{\Lambda}{\sqrt{e+\epsilon}}\|_{L^\infty(a,b)}\|\d_t\sqrt{\rho}\|_{L^2(a,b)}\\
\le & 4 \|f'(\sqrt{\rho})\sqrt{\rho}\|_{L^2}+2\|\lambda\|_{L^2}+2\|\d_t\sqrt{\rho}\|_{L^2}\\
\le & C(M_1,M_2).
\end{align*}
Passing to the limit as $\epsilon\to 0$, it follows that
\[
\|\d_x\sqrt{e}\|_{L^2(a,b)}\le C(M_1,M_2).
\]
Last by choosing $x_0\in[a,b]$ such that $e(x_0)\le 2\, \underset{x\in(a,b)}{\inf}\, e(x)$, we can directly compute that
\begin{align*}
\|\sqrt{e}\|_{L^\infty(a,b)}^2=&\|e\|_{L^\infty(a,b)}\leq e(x_0)+\int_a^b \d_x e(x)\,dx\\
\leq & 2\, \underset{x}{\inf}\, e(x)+2\|\sqrt{e}\|_{L^2(a,b)}\|\d_x\sqrt{e}\|_{L^2(a,b)}\\
\leq & 2 |b-a|^{-1} \|e\|_{L^1(a,b)}+2\|e\|_{L^1(a,b)}^\frac12\|\d_x\sqrt{e}\|_{L^2(a,b)}\\
\leq & C(M_1,M_2)(1+|b-a|^{-1}).
\end{align*}
\end{proof}

In the next Proposition we will prove that $\d_x^2\sqrt{\rho}$ and $\frac{\Lambda^2}{\sqrt{\rho}}$ are integrable on any open interval $(a,b)$ outside the vacuum, including the case when $a$ and $b$ are at vacuum boundaries.
\begin{prop}\label{prop:integrability}
Let $(\sqrt{\rho}, \Lambda)$ be a GCP state satisfying the bounds \eqref{eq:GCP_bounds}. Then on any interval $(a,b)\subset V^c$ with finite length, we have $\d_x^2\sqrt{\rho}\in L^1(a,b)$, $\frac{\Lambda^2}{\sqrt{\rho}}\in L^1(a,b)$, and the following bound holds
\[
\|\d_x^2\sqrt{\rho}\|_{L^1(a,b)}+\|\frac{\Lambda^2}{\sqrt{\rho}}\|_{L^1(a,b)}\le C(M_1,M_2)(1+|b-a|^{\frac12}+|b-a|^{-\frac12}).
\]
\end{prop}

\begin{proof}
We first take a function $\eta(x)$ on $[0,2]$ such that $0\le\eta(x)\le 1$, $\eta(x)\equiv 0$ for $x\in[0,\frac14]$, $\eta(x)\equiv 1$ for $x\in[\frac74,2]$ and $|\eta'(x)|\le 1$ ($\eta$ is easy to construct by mollifying a piecewise linear function). For $0<\epsilon<|b-a|/4$, we define 
\[
\eta_\epsilon(x)=\left\{\begin{array}{ll}
\eta\left(\frac{x-a}{\epsilon}\right) & x\in(a,a+2\epsilon]\\
1 & x\in [a+2\epsilon,b-2\epsilon]\\
\eta\left(\frac{b-x}{\epsilon}\right) & x\in [b-2\epsilon,b)\\
\end{array}
\right.\in C^\infty_c(a,b).
\]
Then $\eta'_\epsilon$ is supported in $(a,a+2\epsilon)$ and $(b-2\epsilon,b)$, and $|\eta'_\epsilon(x)|\le \frac{1}{\epsilon}$.

Multiplying \eqref{eq:def_lambda} by $2\eta_\epsilon$ and integrating over $(a,b)$, we obtain 
\begin{align}\label{eq:eq5}
\int_{a}^{b}\eta_\epsilon\frac{\Lambda^2}{\sqrt{\rho}}dx=2\int_{a}^{b}\eta_\epsilon\lambda\,dx+\int_{a}^{b}\eta_\epsilon\d_x^2\sqrt\rho\,dx-2\int_{a}^{b}\eta_\epsilon f'(\rho)\sqrt\rho\,dx.
\end{align}
In the right hand side of \eqref{eq:eq5}, the first and the third integrals can be controlled as 
\begin{align*}
&|\int_{a}^{b}\eta_\epsilon\lambda\,dx|\le \|\eta_\epsilon\|_{L^2(a,b)}\|\lambda\|_{L^2(\R)}\le C(M_2)|b-a|^\frac12,\\
&|\int_{a}^{b}\eta_\epsilon f'(\rho)\sqrt\rho\,dx|\le  \|f'(\rho)\sqrt\rho\|_{L^1(\R)}\le C(M_1).
\end{align*}
For the second integral in the right hand side of \eqref{eq:eq5}, applying integration by parts and using $\eta_\epsilon\in C_c^\infty(a,b)$, we obtain
\begin{align}\label{eq:eq6}
\int_{a}^{b}\eta_\epsilon(x)\d_x^2\sqrt\rho\,dx=\int_{a}^{b}\eta'_\epsilon(x)\d_x\sqrt\rho\,dx.
\end{align}
By our construction and Proposition \ref{prop:inftybd}, \eqref{eq:eq6} is bounded by 
\begin{align*}
\left|\eqref{eq:eq6}\right|\le & \frac{|supp(\eta'_\epsilon)|}{\epsilon}\|\d_x\sqrt{\rho}\|_{L^\infty_x(a,b)}\\
\le & 4\sqrt{2}\|\sqrt{e}\|_{L^\infty(a,b)}\leq C(M_1,M_2)(1+|b-a|^{-\frac12}).
\end{align*}
Since $\eta_\epsilon$ and $\frac{\Lambda^2}{\sqrt{\rho}}$ are non-negative functions, we obtain by the estimate of \eqref{eq:eq5} that 
\begin{align*}
\int_{a+2\epsilon}^{b-2\epsilon}\frac{\Lambda^2}{\sqrt{\rho}}dx\le& \int_{a}^{b}\eta_\epsilon\frac{\Lambda^2}{\sqrt{\rho}}dx\\
\le& C(M_1,M_2)(1+|b-a|^{\frac12}+|b-a|^{-\frac12}).
\end{align*}
Passing to the limit as $\epsilon\to 0$ by Fatou's Lemma, we conclude 
\[
\int_{a}^{b}\frac{\Lambda^2}{\sqrt{\rho}}dx\le C(M_1,M_2)(1+|b-a|^{\frac12}+|b-a|^{-\frac12}).
\]
To obtain the integrability of $\d_x^2\sqrt\rho$ we simply notice that by definition of $\lambda$, in the interval $(a,b)$ we have
\[
\d_x^2\sqrt\rho=\frac{\Lambda^2}{\sqrt{\rho}}-2\lambda+f'(\rho)\sqrt{\rho}.
\]
Therefore $\|\d_x^2\sqrt\rho\|_{L^1(a,b)}$ can be controlled by
\begin{align*}
\|\d_x^2\sqrt{\rho}\|_{L^1(a,b)}\le& \|\frac{\Lambda^2}{\sqrt\rho}\|_{L^1(a,b)}+|b-a|^\frac12\|\lambda\|_{L^2_x}+\|f'(\rho)\sqrt{\rho}\|_{L^1_x}\\
\le & C(M_1,M_2)(1+|b-a|^{\frac12}+|b-a|^{-\frac12}).
\end{align*}
\end{proof}

As a consequence of Proposition \ref{prop:inftybd} and Proposition \ref{prop:integrability}, the local integrability of $\d_x\sqrt{e}$ and $\d_x^2\sqrt{\rho}$ imply $\sqrt{e}$ and $\d_x\sqrt{\rho}$ are absolutely continuous functions, therefore it can be continuously extended to the vacuum boundary. More precisely we have the following Corollary (where $\mathcal{AC}$ denotes the space of Absolutely Continuous functions)

\begin{corollary}
For all connected components $(a_j,b_j)\subset V^c$ as given in \eqref{eq:nonvac} with finite length, we have $\sqrt{e}\in \mathcal{AC}(a_j,b_j)$ and $\d_x\sqrt{\rho}\in \mathcal{AC}(a_j,b_j)$, therefore we can continuously extend $\sqrt{e}$ and $\d_x\sqrt{\rho}$ to functions in $\mathcal{C}[a_j,b_j]$.

For the semi-infinite interval $(-\infty,b_{-\infty})$ and $(a_\infty,\infty)$, $\sqrt{e}$ and $\d_x\sqrt{\rho}$ belong to $\mathcal{AC}[-R,b_{-\infty}]$ or $\mathcal{AC}[a_\infty,R]$ respectively, for any finite $R>|a_\infty|,|b_{-\infty}|$.
\end{corollary}
\begin{proof}
The statement of absolute continuity directly comes from the integrability of $\d_x\sqrt{e}$ and $\d_x^2\sqrt{\rho}$ on $(a_j,b_j)$. Moreover we can directly define the continuous boundary value from right as 
\begin{align*}
\sqrt{e}(t,a_j^+)=&\sqrt{e}(x_j)+\int_{x_j}^{a_j}\d_x\sqrt{e}(y)dy,\\
\d_x\sqrt{\rho}(t,a_j^+)=&\d_x\sqrt{\rho}(x_j)+\int_{x_j}^{a_j} \d_x^2\sqrt{\rho}(y)dy,
\end{align*}
where $x_j\in (a_j,b_j)$ is a fixed point. Similarly we can define the left boundary value $\sqrt{e}(t,b_j^-)$ and $\d_x\sqrt{\rho}(t,b_j^-)$.

Same argument applies to intervals $(-R,b_{-\infty})$ and $(a_\infty,R)$.
\end{proof}

Now we come back to the definition \eqref{eq:def_lambda} of $\lambda$. In the right hand side of \eqref{eq:def_lambda} we proved that $\d_x^2\sqrt{\rho}$ and $\Lambda^2/\sqrt{\rho}$ belongs to $L^1(a_j,b_j)$ for all components $(a_j,b_j)\subset V^c$. Therefore the only possible singularities are given by the jumps of $\d_x\sqrt{\rho}$ at the vacuum boundaries, which leads to a Dirac-delta in $\d_x^2\sqrt{\rho}$ in the sense of distribution, and it should be balanced to make $\lambda$ a function in $L^2_x(\R)$. Hence in the case that the vacuum boundaries has no accumulation point, we can extend $\lambda$ to a Radon measure on $\R$. We conclude it by the next Proposition.

\begin{prop}\label{prop:lambda}
Let $(\sqrt{\rho}, \Lambda)$ be a GCP state satisfying the bounds \eqref{eq:GCP_bounds} and let us assume that the vacuum boundaries $\underset{j}{\cup}\{a_j,b_j\}$ given in \eqref{eq:nonvac} have no accumulation point.
Then we have that 
\begin{equation*}
\tilde\lambda=-\frac12\d_x^2\sqrt{\rho}+\frac{\Lambda^2}{\sqrt\rho}\mathbf{1}_{\{\rho>0\}}+f'(\rho)\sqrt{\rho}
\end{equation*}
is a Radon measure in $\mathcal{M}(\R) $ and $\d_x^2\sqrt{\rho}$ is a distribution in $H^{-1}(\R)$. 
\\ Therefore it follows that the measure $\lambda dx$ extends to the Radon measure $\tilde\lambda$ in the following way
\begin{equation}\label{eq:tildelambda}
\tilde\lambda=\lambda\,dx+\frac12\underset{j}{\sum}\left(\d_x\sqrt{\rho}(b_j^-)\delta_{b_j}-\d_x\sqrt{\rho}(a_j^+)\delta_{a_j}\right),
\end{equation} 
where $\delta_{\cdot}$ is the Dirac-delta distribution.
\end{prop}

\begin{proof}
We only need to prove the identity \eqref{eq:tildelambda}. Let $\eta\in\mathcal{C}^\infty_c(\R)$ be any test function, by \eqref{eq:def_lambda} 
\begin{align*}
\int_\R \eta\,\lambda\,dx=&\int_{V^c}\eta\,\lambda\,dx\\
=&\underset{j}{\sum}\int_{a_j}^{b_j}\eta\left(-\frac12\d_x^2\sqrt\rho+\frac12\frac{\Lambda^2}{\sqrt\rho}+f'(\rho)\sqrt{\rho}\right)dx.
\end{align*}
The integrals in the right hand side of the last identity are well defined by Proposition \ref{prop:integrability}, and the summation is taken over finite $j$ since $\{a_j\}$, $\{b_j\}$ have no accumulation point and the support of $\eta$ is compact. By using $\d_x\sqrt{\rho}\in \mathcal{C}[a_j,b_j]$ and integration by parts we get 
\begin{align*}
-\frac12\int_{a_j}^{b_j}\eta\d_x^2\sqrt\rho\,dx=-\frac12\d_x\sqrt\rho(b_j^-)\eta(b_j)+\frac12\d_x\sqrt\rho(a_j^+)\eta(a_j)+\frac12\int_{a_j}^{b_j}\d_x\eta\d_x\sqrt\rho\,dx.
\end{align*}
Summing over $j$ implies 
\begin{align*}
-\frac12\underset{j}{\sum}\int_{a_j}^{b_j}\eta\d_x^2\sqrt\rho\,dx=\int_{V^c}\d_x\eta\d_x\sqrt\rho\,dx+\frac12\underset{j}{\sum}\left(\d_x\sqrt\rho(a_j^+)\eta(a_j)-\d_x\sqrt\rho(b_j^-)\eta(b_j)\right).
\end{align*}
By Lemma \ref{lemma:LL} we have $\d_x\sqrt{\rho}=0$ a.e. on $V=\{\rho=0\}$, thus 
\[
\int_{V^c}\d_x\eta\d_x\sqrt\rho\,dx=\int_{\R}\d_x\eta\d_x\sqrt\rho\,dx=-\int_{\R}\eta\d_x^2\sqrt\rho\,dx.
\]
At the same time we also have 
\begin{align*}
\underset{j}{\sum}\int_{a_j}^{b_j}\eta\frac{\Lambda^2}{\sqrt{\rho}}\,dx=&\int_{\R}\eta\frac{\Lambda^2}{\sqrt\rho}\mathbf{1}_{\{\rho>0\}}dx,\\
\underset{j}{\sum}\int_{a_j}^{b_j}\eta\,f'(\rho)\sqrt{\rho}\,dx=&\int_{\R}\eta\,f'(\rho)\sqrt{\rho}\,dx.
\end{align*}
Therefore we conclude 
\begin{align*}
\int_{\R}\eta\,\lambda\,dx=&\int_{\R}\eta\left(-\frac12\d_x^2\sqrt\rho+\frac12\frac{\Lambda^2}{\sqrt\rho}\mathbf{1}_{\{\rho>0\}}+f'(\rho)\sqrt{\rho}\right)dx\\
&+\frac12\underset{j}{\sum}\left(\d_x\sqrt\rho(a_j^+)\eta(a_j)-\d_x\sqrt\rho(b_j^-)\eta(b_j)\right)\\
=&\left<\eta,\tilde\lambda\right>+\left<\eta,\frac12\underset{j}{\sum}\left(\d_x\sqrt{\rho}(a_j^+)\delta_{a_j}-\d_x\sqrt{\rho}(b_j^-)\delta_{b_j}\right)\right>
\end{align*}
for any test function $\eta\in\mathcal{C}_c^\infty(\R)$, which concludes the proof
\end{proof}

\begin{rem}
The identity \eqref{eq:tildelambda} shows $\lambda dx$ is the absolutely continuous part of the Radon measure $\tilde\lambda$ with respect to Lebesgue measure, and the singular part of $\tilde\lambda$ is a sequence of Dirac-delta located at vacuum boundaries. Therefore $\tilde\lambda$ and $\lambda$ agree at the level of $\xi$ in the sense that 
\[
\xi=\sqrt{\rho}\tilde\lambda=\sqrt{\rho}\lambda.
\]

On the other hand, since $\tilde\lambda$ takes into account the possible jumps of $\d_x\sqrt\rho$ at vacuum boundaries, our conjecture is that $\tilde\lambda$ might be a useful tool to characterise the behaviour of vacuum boundaries or other functions, for example the polar factor or phase function, whose singularity is at vacuum boundaries. Indeed formally by using the WKB ansatz $\psi=\sqrt{\rho}e^{iS}$ and the Schr\"odinger equation, one has
\[
\d_t S+\mu=0,
\]
where $\mu$ is the chemical potential
\[
\mu=-\frac{\d_x^2\sqrt{\rho}}{2\sqrt\rho}+\frac12 v^2+f'(\rho).
\]
In the general case a rigorous definition of the phase function $S$ and chemical potential $\mu$ is not possible, however the previous equation above suggests the possibility to reconstruct the phase function $S$ from hydrodynamic data by using the identity
\[
\sqrt{\rho}\d_t S+\tilde\lambda=0.
\]
\end{rem}

\section*{Acknowledgments}
The first and the second author have been partially supported by PRIN-MIUR project 2015YCJY3A\_003 \emph{Hyperbolic Systems of Conservation Laws and Fluid Dynamics: Analysis and Applications}. The first author acknowledges partial support through the INdAM-GNAMPA project \emph{Esistenza, limiti singolari e comportamento asintotico per equazioni Eulero/Navier– Stokes–Korteweg}.


\begin{thebibliography}{100}
\bibitem{AHMZ} P. Antonelli, L.E. Hientzsch, P. Marcati, H. Zheng \emph{On some results for quantum hydrodynamical models}, Mathematical Analysis in Fluid and Gas Dynamics, \href{http://www.kurims.kyoto-u.ac.jp/~kyodo/kokyuroku/contents/2070.html}{Proceeding RIMS K\^oky\^uroku} {\bf 2070} (2018), 107--129.
\bibitem{AM1} P. Antonelli, P. Marcati, \emph{On the finite energy weak solutions to a system in Quantum Fluid Dynamics}, Comm. Math. Phys. {\bf 287} (2009), no 2, 657--686.
\bibitem{AM2} P. Antonelli, P. Marcati, \emph{The Quantum Hydrodynamics system in two space dimensions}, Arch. Rat. Mech. Anal. \textbf{203} (2012), 499--527.
\bibitem{AMDCDS} P. Antonelli, P. Marcati, \emph{Quantum hydrodynamics with nonlinear interactions}, Disc. Cont. Dyn. Sys. Ser. S {\bf 9}, no. 1 (2016), 1--13.
\bibitem{AM3} P. Antonelli, P. Marcati, \emph{Some results on systems for quantum fluids}, (2016), Recent Advances in Partial Differential Equations and Application, Cont. Math. {\bf666}, 41--54.
\bibitem{AMZ2} P. Antonelli, P. Marcati, H. Zheng, \emph{An intrinsically hydrodynamic approach to multidimensional QHD systems}, (2016), reprint \url{http://arxiv.org/abs/1912.05448}
\bibitem{AS} P. Antonelli and S. Spirito, {\em On the compactness of finite energy weak solutions to the quantum Navier-Stokes equations}, J. Hyperbolic Differ. Equ., \textbf{15} (2018), 133--147.
\bibitem{AS1} P. Antonelli and S. Spirito, \emph{Global existence of finite energy weak solutions of quantum Navier-Stokes equations}, Arch. Ration. Mech. Anal., {\bf3} (2017), 1161--1199.
\bibitem{AH} C. Audiard, B. Haspot, \emph{Global well-posedness of the Euler-Korteweg system for small irrotational data}, Comm. Math. Phys. {\bf 351}, no. 1 (2017), 2017--247.
\bibitem{Bar} J. Barab ,\emph{Nonexistence of asymptotically free solutions for a nonlinear Schr\"odinger equation}, J. Math. Phys. {\bf 25} (1984), 3270.
\bibitem{BD} C. Bardos, P. Degond, \emph{Global existence for the Vlasov-Poisson equation in 3 space variables with small initial data}, Ann. I.H.P. sect. C {\bf 2}, no. 2 (1985), 101--118.
\bibitem{BG} S. Benzoni-Gavage, \emph{Propagating phase boundaries and capillary fluids}, available online at \url{http://math.univ-lyon1.fr/~benzoni/Levico.pdf}.
\bibitem{BGDD} S. Benzoni-Gavage, R. Danchin, S. Descombes, \emph{On the well-posedness for the Euler-Korteweg model in several space dimensions}, Indiana Univ. Math. J. {\bf 56} (2007), 1499--1579.
\bibitem{Caz} T. Cazenave, \emph{Semilinear Schr\"odinger Equations}. Courant Lecture Notes in Mathematics vol. 10, New York University, Courant Institute of Mathematical Sciences, AMS, 2003.
\bibitem{CDS} R. Carles, R. Danchin, J.-C. Saut, \emph{Madelung, Gross-Pitaevskii and Korteweg}, Nonlinearity {\bf 25} (2012), 2843--2873.

\bibitem{CG} R. Carles, I. Gallagher, \emph{Universal dynamics for the defocusing logarithmic Schr\"odinger equation}, Duke Math. J. {\bf 167}, no. 9 (2018), 1761--1801.
\bibitem{CCH} R. Carles, K. Carrapatoso, M. Hillairet, \emph{Rigidity results in generalized isothermal fluids}, Ann. H. Leb. {\bf 1} (2018), 47--85.
\bibitem{Chem} J.Y. Chemin, \emph{Dynamique des gas \`a masse totale finie}, Asympt. Anal. {\bf 3} (1990), 215--220.


\bibitem{CGT} J. Colliander, M. Grillakis, N. Tzirakis, \emph{Tensor products and correlation estimates with applications to nonlinear Schr\"odinger equations}, Comm. Pure Appl. Math. {\bf 62}, no. 7 (2009), 920--968.
\bibitem{CKSTT} J. Colliander, M. Keel, G. Staffilani, H. Takaoka, T. Tao, \emph{Global existence and scattering for rough solutions of a nonlinear Schr\"odinger equation on $\R^3$}, Comm. Pure Appl. Math. {\bf 57}, no. 8 (2004), 987--1014.
\bibitem{TzavFasc} T. Debiec, P. Gwiazda, A. Swierczewska-Gwiazda, A. Tzavaras, \emph{Conservation of energy for the Euler-Korteweg equations}, Calc. Var. and PDEs {\bf 57}, no.6:Art 160, 12 (2018).
\bibitem{Daf} C. Dafermos, \emph{Hyperbolic conservation laws in continuum physics}, Springer 2005.
\bibitem{DFM}D. Donatelli, E. Feireisl, P. Marcati,\emph{ Well/ill posedness for the Euler- Korteweg-Poisson system and related problems}, Communications in Partial Differential Equations, {\bf 40} (2015), 1314--1335.
\bibitem{FJS} D. Fajman, J. Joudioux, J. Smulevici, \emph{A vector field method for relativistic transport equations with applications}, Anal. PDE {\bf 10}, no. 7 (2017), 1539--1612.
\bibitem{Gar} C. Gardner, \textit{The quantum hydrodynamic model for semincoductor devices}, SIAM J. Appl. Math. {\bf 54} (1994), 409--427.
\bibitem{GM} I. Gasser, P. Markowich, \emph{Quantum hydrodynamics, Wigner transforms and the classical limit}, Asympt. Anal. {\bf 14} (1997), 97--116.
\bibitem{GLT} J. Giesselmann, C. Lattanzio, A. Tzavaras, \emph{Relative energy for the Korteweg theory and related Hamiltonian flows in gas dynamics}, Arch. Ration. Mech. Anal. {\bf 223} (2017), no. 3, 1427--1484.
\bibitem{GT} J. Giesselmann, A. Tzavaras, \emph{Stability properties of the Euler-Korteweg system with nonmonotone pressures}, Appl. Anal. {\bf 96} (2017), no. 9, 1528--1546.
\bibitem{GV} J. Ginibre, G. Velo, \emph{On a class of nonlinear Schr\"odinger equations. II. Scattering theory, general case}, J. Funct. Anal. {\bf 32} (1979), 33--71.
\bibitem{GVQ} J. Ginibre, G. Velo, \emph{Quadratic Morawetz inequalities and asymptotic completeness in the energy space for nonlinear Schr\"odinger and Hartree equations}, Quart. Appl. Math. {\bf 68}, no. 1 (2010), 113--134. 
\bibitem{Gl} R. Glassey, \emph{On the blowing up of solutions to the Cauchy problem for nonlinear Schr\"odinger equations}, J. Math. Phys. {\bf 18}, no. 9 (1977), 1794--1979.
\bibitem{HNT1} N. Hayashi, K. Nakamitsu, M. Tsutsumi, \emph{On solutions of the initial value problem for the nonlinear Schr\"odinger equations in one space dimension}, Math. Z. {\bf 192} (1986), 637--650.
\bibitem{HNT2} N. Hayashi, K. Nakamitsu, M. Tsutsumi, \emph{On solutions of the initial value problem for the nonlinear Schr\"odinger equations}, J. Funct. Anal. {\bf 71} (1987), 218--245.
\bibitem{IR} R. Illner, G. Rein, \emph{Time decay of the solutions of the Vlasov-Poisson system in the plasma physical case}, Math. Methods Appl. Sci. {\bf 6} (1984), 1409--1413.
\bibitem{HLM} F. Huang, H.-L. Li, A. Matsumura, \emph{Existence and stability of steady-state of one-dimensional quantum hydrodynamic system for semiconductors}, J. Diff. Equ. {\bf 225} (2006), 1--25.
\bibitem{HLMO} F. Huang, H.-L. Li, A. Matsumura, S. Odanaka, \emph{Well-posedness and stability of quantum hydrodynamics ofr semiconductors in $\mathbb R^3$}, in ``Some problems on nonlinear hyperbolic equations and applications'', Ser. Contemp. Appl. Math. CAM, vol. 15, p. 131--160. Higher Ed. Press, Beijing, 2010.
\bibitem{John} F. John, \emph{Nonlinear wave equations. Formation of singularities}, Univ. Lecture Ser., Amer. Math. Soc., Providence, RI, 1990.
\bibitem{J} A. J\"ungel, \emph{Dissipative quantum fluid models}, Riv. Mat. Univ. Parma {\bf 3} (2012), 217--290.
\bibitem{JqNS} A. J\"ungel, \emph{Global weak solutions to compressible Navier-Stokes equations for quantum fluids}, SIAM J. Math. Anal. {\bf 42}, no. 3 (2010), 1025--1045.
\bibitem{JLM} A. J\"ungel, H.-L. Li, A. Matsumura, \emph{The relaxation-time limit in the quantum hydrodynamic equations for semiconductors}, J. Diff. Equ. {\bf 225}, no. 2 (2006), 440--464.
\bibitem{JLMG} A. J\"ungel, J.-L. L\'opez, J. Montejo-G\'amez, \emph{A new derivation of the quantum Navier-Stokes equations in the Wigner-Fokker-Planck approach}, J. Stat. Phys. {\bf 145}, no. 6 (2011), 1661--1673.
\bibitem{JMR} A. J\"ungel, M.C. Mariani, D. Rial, \emph{Local existence of solutions to the transient quantum hydrodynamic equations}, Math. Mod. Meth. Appl. Sci. {\bf 12} (2002), 485.
\bibitem{Khal} I. Khalatnikov, \emph{An Introduction to the Theory of Superfluidity}, 2000.
\bibitem{Kl} S. Klainerman, \emph{Uniform decay estimates and the Lorentz invariance of the classical wave equation}, Comm. Pure Appl. Math. {\bf 38}, no. 3 (1985), 321--332.
\bibitem{LLX} H.-L. Li, J. Li, Z. Xin, \emph{Vanishing of vacuum states and blow-up phenomena of the compressible Navier–Stokes equations}, Commun, Math. Phys. {\bf 281} (2008) 401--444.
\bibitem{LM} H-L. Li, P. Marcati, \emph{Existence and asymptotic behavior of multi-dimensional quantum hydrodynamic model for semiconductors}, Comm. Math. Phys. {\bf 245} (2004), 215--247.
\bibitem{LX} J. Li, Z. Xin, \emph{Global existence of weak solutions to the barotropic compressible Navier-Stokes flows with degenerate viscosities}, archived as \url{arxiv.org/abs/1504.06826}
\bibitem{LL} E. Lieb, M. Loss, \textit{Analysis}, Graduate Studies in Mathematics, vol. 14, AMS, 2001.
\bibitem{LP} F. Linares and G. Ponce, {\it Introduction to nonlinear dispersive equations}. Springer-Verlag, New York, 2009.
\bibitem{LS} J-E. Lin and W. A. Strauss, {\it Decay and scattering of solutions of a nonlinear Schrödinger equation}. Journal of Functional Analysis {\bf 30}, Issue 2, (1978), 245-263.
\bibitem{Mad} E. Madelung, \textit{Quantuentheorie in hydrodynamischer form}, Z. Physik {\bf 40} (1927), 322.
\bibitem{MS} P. Markowich, J. Sierra, \emph{Non-uniqueness of weak solutions of the Quantum-Hydrodynamic system}, Kin. Rel. Models {\bf 12}, no. 2 (2019), 347--356.
\bibitem{MV} C. Mouhot, C. Villani, \emph{On Landau damping}, Acta Math. {\bf 207}, no. 1 (2011), 29--201.
\bibitem{Mor} C. Morawetz, \emph{Time decay for the nonlinear Klein-Gordon equations}, Proc. Roy. Soc. Ser. A {\bf 306} (1968), 291--296.
\bibitem{Pe} B. Perthame, \emph{Time decay, propagation of low moments and dispersive effects for kinetic equations}, Comm. PDEs {\bf 21} (1996), 659--686.
\bibitem{PS} L. Pitaevskii, S. Stringari, \emph{Bose-Einstein condensation and superfluidity}, Clarendon Press, Oxford, 2016.
\bibitem{PTV} F. Planchon, N. Tzvetkov, N. Visciglia, \emph{On the growth of Sobolev norms for NLS on 2-and 3-dimensional manifolds}, Anal. PDEs {\bf 10}, no. 5 (2017), 1123--1147.
\bibitem{PV} F. Planchon, L. Vega, \emph{Bilinear virial identities and applications}, Ann. Sci. E.N.S., S\'erie 4 {\bf 42}, no. 2 (2009), 261--290.
\bibitem{Roz} O. Rozanova, \emph{Hydrodynamic approach to constructing solutions of the nonlinear Schr\"odinger equation in the critical case}, Proc. Amer. Math. Soc. {\bf 133}, no. 8 (2005), 2347--2358.
\bibitem{Ser} D. Serre, \emph{Solutions classiques globales des \'equations d'Euler pour un fluide parfait compressible}, Ann. Inst. Four. {\bf 47}, no. 1 (1997), 139--153.
\bibitem{Sid} T. Sideris, \emph{Formation of singularities in three-dimensional compressible fluids}, Comm. Math. Phys. 
{\bf 101}, no. 4 (1985), 475--485.
\bibitem{Straus} W. Strauss, \emph{Nonlinear invariant wave equations}, Invariant wave equations (Proc. ``Ettore Majorana'' Internat. School of Math. Phys., Erice, 1977), pp. 197--249, Lecture Notes in Phys., 73, Springer, Berlin-New York, 1978. 
\bibitem{Tao} T. Tao, \emph{Nonlinear Dispersive Equations: Local and Global Analysis}.
CBMS Regional Conference Series in Mathematics, AMS 2006.
\end{thebibliography}
\end{document}